\documentclass[11pt]{article}
\usepackage[margin = 1 in]{geometry}

\usepackage{amsmath, amsthm, amssymb, mathrsfs, dsfont, bm, mathtools, xcolor, verbatim, bbm}
\usepackage[colorlinks = true, citecolor=blue]{hyperref}
\allowdisplaybreaks

\newtheorem{theorem}{Theorem}[section]
\newtheorem{lemma}[theorem]{Lemma}
\newtheorem{prop}[theorem]{Proposition}
\numberwithin{equation}{section}

\renewcommand{\P}{\mathbb{P}}					
\newcommand{\sP}{\mathsf{P}}					
\newcommand{\bP}{\mathbf{P}}					
\newcommand{\C}{\mathbb{C}}						
\newcommand{\cC}{\mathcal{C}}					
\newcommand{\Z}{\mathbb{Z}}						
\newcommand{\N}{\mathbb{N}}						
\newcommand{\E}{\mathbb{E}}						
\newcommand{\R}{\mathbb{R}}						
\newcommand{\D}{\mathcal{D}}					
\newcommand{\A}{\mathcal{A}}					
\newcommand{\bE}{\mathbf{E}}					
\newcommand{\tp}{\mathtt{p}}                    

\newcommand{\diag}{\mathrm{diag}}				
\newcommand{\dl}{\mathcal{L}}					
\newcommand{\dks}{\mathcal{KS}}					
\newcommand{\f}{\mathrm{F}}						
\newcommand{\op}{\mathrm{op}}					
\newcommand{\cE}{\mathcal{E}}				 	
\newcommand{\supp}{\mathrm{supp}}				
\newcommand{\ct}{\mathscr{T}}					
\newcommand{\cc}{\mathscr{C}}					

\newcommand{\M}{\text{\it \scalebox{.7}{M}}}
\renewcommand{\L}{\text{\it \scalebox{.7}{L}}}

\newcommand{\MK}{\text{\it \scalebox{.7}{M,K}}}
\newcommand{\ML}{\text{\it \scalebox{.7}{M,L}}}
\newcommand{\MLK}{\text{\it \scalebox{.7}{M,L,K}}}

\newcommand{\fH}{H^\M}
\newcommand{\sH}{H^\ML}
\newcommand{\tH}{H^\MLK}
\newcommand{\fD}{\Delta^\M}
\newcommand{\sD}{\Delta^\ML}
\newcommand{\tD}{\Delta^\MLK}
\newcommand{\dD}{\Delta^\text{\it \scalebox{.7}{M,K}}}

\newcommand{\lp}{\left(}						
\newcommand{\rp}{\right)}						
\newcommand{\bb}[1]{\Big\lbrace #1 \Big\rbrace}	

\newcommand{\1}{\mathds{1}}                             
\newcommand{\lz}{\ell^2(\mathbb{Z})}                    
\newcommand{\sgn}{\mathrm{sgn}}                         
\newcommand{\ed}{\stackrel{d}{=}}                       

\begin{document}
\title{Limiting eigenvalue distribution of heavy-tailed Toeplitz matrices}
\author{Ratul Biswas and Arnab Sen}
\date{}
\maketitle
\begin{abstract}
We consider an $N\times N$ random symmetric Toeplitz matrix with an i.i.d.\ input sequence drawn from a distribution that lies in the domain of attraction of an $\alpha$-stable law for $0 < \alpha < 2$. We show that under an appropriate scaling, its empirical eigenvalue distribution, as $N \to \infty$, converges weakly to a random symmetric probability distribution on $\R$, which can be described as the expected spectral measure of a certain random unbounded self-adjoint operator on $\ell^2(\Z)$. The limiting distribution turns out to be  almost surely subgaussian. Furthermore,  the support of the limiting distribution is bounded almost surely if $0<\alpha <1$ and is unbounded almost surely if  $1\leq \alpha <2$.
\end{abstract}

\section{Introduction}

Toeplitz matrices are ubiquitous in a variety of fields of mathematics and statistics including signal processing, numerical methods, machine learning, and time series. They have a rich, well-developed mathematical theory \cite{Bot2, Bot1}.  A (symmetric) Toeplitz matrix with i.i.d.\ input sequence provides a non-mean-field random matrix model which is important both from the theoretical and applied viewpoints. Despite this, the efforts to understand the spectral behavior of random Toeplitz matrices have received limited success so far, especially when compared to the classical mean-field models like Wigner matrices.

Answering a question posed in \cite{Bai2}, the article \cite{Bry} proved that the eigenvalues of a random Toeplitz matrix with input sequence having unit variance have a non-random limiting distribution that does not depend on the specific choice of the entry distribution, see also \cite{Ham, Bos2}. In \cite{Sen2},  it was shown, under the finite $(2+\varepsilon)$-moment assumption, that the maximum eigenvalue of the random Toeplitz matrix of size $n$, scaled by $\sqrt{n \log n}$, converges to a constant that is related to the $2\to4$ norm of operator norm of the sine kernel. The fluctuation of the linear statistics of the eigenvalues of a random Toeplitz for polynomial test functions is known to obey a CLT, see \cite{Cha, Liu}.

Obtaining finer spectral properties of the random Toeplitz matrices remains a mathematical challenge. Indeed, the resolvent analysis that can yield very precise results on the local and global spectral statistics for the mean-field models becomes ineffective in the Toeplitz model due to the lack of full independence and Toeplitz structure. The proof of the existence of the limiting eigenvalue distribution of random Toeplitz matrices in \cite{Bry} was based on the method of moments and as such the limiting measure does not have an explicit description. In fact, we know only very basic information about this distribution and  almost nothing is known beyond the fact that it has a bounded density which was shown in \cite{Sen1}. 

In the world of random matrices, the asymptotic spectral statistics do not generally depend on the choice of i.i.d.\ entry distribution under the finite variance assumption. So, it is a natural question to investigate what happens beyond finite variance, say when the entries are drawn i.i.d.\ from a heavy-tailed distribution with infinite variance. The case of Wigner matrices with heavy-tailed entries with tail index $\alpha \in (0, 2)$ (also, known as L\'evy matrices) is relatively well-studied. After the pioneering (non-rigorous) work \cite{Ciz}, it was shown rigorously in  \cite{Ben} using resolvent methods and in \cite{Zak} using the method of moments that the limiting eigenvalue distribution is a non-random symmetric distribution that only depends on $\alpha$. See also \cite{Bel} and \cite{Bor3}. The limiting distribution is also heavy-tailed (with unbounded support) with the same tail index. The eigenvectors of the Levy matrices have attracted a lot of attention lately \cite{Bor1, Bor2, Agg2, Agg1}. When $1< \alpha < 2$, the eigenvectors are known to be completely delocalized. For $0< \alpha<1$, the model exhibits Anderson localization-type phase transition, where the eigenvectors with eigenvalues (in magnitude) below a threshold, called the mobility edge, are completely delocalized and the eigenvectors with eigenvalues above the threshold are localized. 

In this article, we establish the limiting eigenvalue distribution of the heavy-tailed Toeplitz matrix. The limiting distribution of the symmetric circulant matrices with heavy-tailed entries was studied in \cite{Bos1}. For circulants, the eigenvalues are just the discrete Fourier transform of the first row. The eigenvalues of Toeplitz matrices, however, lack such explicit representation. We employ the standard method of moments, not on the original matrix but the one obtained by conjugating it with the discrete Fourier transform matrix. Since we are working with heavy-tailed entries, we truncate them at an appropriate level so that the truncated random variables have all finite moments. Our key observation is that after the Fourier conjugation, the heavy-tailed Toeplitz matrix is well-approximated by a random band matrix where the width of the band depends on the approximation level. This allows us to find a limiting random operator on $\ell^2(\Z)$, whose law only depends on $\alpha$, and the eigenvalue limiting distribution is connected to the spectral measure of this operator. 
In contrast to the light-tailed case, the limiting measure for heavy-tailed Toeplitz turns out to be random. The source of this randomness is the magnitudes and locations of the largest entries of the heavy-tailed random variables.
Interestingly, unlike the heavy-tailed Wigner matrix, the limiting measure is not heavy-tailed. In fact, it has a subgaussian tail almost surely. We hope that this limiting operator can further be analyzed to extract more information about the limiting measure.

Next, we describe our model. Let $(\Xi, \mathcal{F},\sP)$ be a probability space on which $(a_k)_{k\geq 0}$ are defined to be i.i.d.\ random variables whose distribution is in the domain of attraction of an $\alpha$-stable law for some $\alpha \in (0,2)$. It is known (\cite{Fel}, Theorem IX.8.1a) that there exists a function $g$ such that \begin{align} \label{eqn 1.1}
g(t) := \sP (|a_0|\geq t) = \frac{\ell(t)}{t^\alpha} \qquad \mathrm{as\; } t\to \infty,
\end{align}
where $\ell:\R_+ \to \R_+$ is a slowly varying function, i.e., it satisfies 
\begin{align*}
\lim_{t\to\infty} \frac{\ell(\beta t)}{\ell(t)} = 1 \qquad \mathrm{for\;all\;}\beta>0.
\end{align*}
We further assume that there exists $0\leq p \leq 1$ such that \begin{align} \label{eqn 1.2}
\lim_{t\to \infty} \frac{\sP ( a_0 \geq t )}{\sP (|a_0| \geq t)} = p \qquad \mathrm{and} \qquad \lim_{t\to \infty} \frac{\sP ( a_0 \leq -t)}{\sP( |a_0| \geq t)} = 1-p.
\end{align}
Define the normalizing constant 
\begin{align*}
c_N: = \inf \lbrace t: \sP(|a_0| \geq t ) \leq N^{-1} \rbrace.
\end{align*} 
It is known that there exists a slowly varying function $\ell_0$ such that $c_N = \ell_0(N) N^{1/\alpha}$.

We consider the $N \times N$  random symmetric Toeplitz matrix $T_N$  with i.i.d.\ heavy-tailed entries scaled by $c_N$ as follows:
 \begin{align*}
T_N=  c_N^{-1} \begin{pmatrix}
a_0 & a_1 & \ldots & a_{N-2} & a_{N-1}\\
a_1 & a_0 & a_1 & \ddots & a_{N-2}\\
\vdots & \ddots & \ddots & \ddots & \vdots\\
a_{N-2} & \ddots & \ddots &  \ddots & a_1\\
a_{N-1} & a_{N-2} & \ldots & a_1 & a_0
\end{pmatrix} = (b_{|k-l|})_{k,l \in [N]},
\end{align*}
where $b_k = c_N^{-1} a_k$ for $k \in [N] : = \{0, \ldots, N-1\}$. 
For a symmetric matrix $A$ of size $N\times N$, we denote by $\mu_A$ the empirical spectral distribution of $A$, i.e., \begin{align*}
\mu_A = \frac{1}{N}\sum_{j=0}^{N-1} \delta_{\lambda_j(A)}
\end{align*}
where $\lambda_0(A) \leq \ldots\leq \lambda_{N-1}(A)$ are the eigenvalues of the matrix $A$ and $\delta_x$ is the Dirac measure at the point $x$.

We are interested in finding the limit of  $\mu_{T_N}$ for a sequence of heavy-tailed Toeplitz matrices. The limit turns out to be a random probability measure, which arises as a spectral measure of a certain random operator on $\ell^2(\Z)$ at some unit vector. Below we describe this random operator. 

 Let $\Pi:\lz \to \lz$ be the projection operator obtained by composing the following operators
\begin{align*}
\Pi: \lz \xrightarrow{\mathfrak{F}}L^2(S^1) \xrightarrow{\1_{[0,1/2]}} L^2(S^1)\xrightarrow{\mathfrak{F}^{-1}}\lz
\end{align*}
where $\mathfrak{F}$ is the Fourier transform, i.e. for a vector $v = (v_n)_{n \in \Z}$, $(\mathfrak{F}v)(x) = \sum_{n \in \Z} v_n e^{ 2 \pi i n x}$  and $\1_{[0,1/2]}$ is the projection that acts by multiplication by the indicator function of the interval $[0,1/2]$. Here, as in the rest of the paper, we use $i$ to denote $\sqrt{-1}$ and never use it as an index. The action of $\Pi$ on the standard basis $(e_k)_{k\in \Z}$ is given by, see \cite{Sen2}, 
\begin{align} \label{eqn 1.3}
\Pi(k,l) : = \langle e_k, \Pi e_l\rangle = \begin{cases} 
\frac{1}{2}1; & \mathrm{\;if\;} k = l,\\
0; & \mathrm{\;if\;}k\neq l \mathrm{\;and\;}|k-l|\text{\;is even},\\
-\frac{i}{\pi(k-l)}; & \mathrm{\;if\;} |k-l|\text{\;is odd}.
\end{cases}
\end{align} 

Let $(\Upsilon, \mathcal{G},\mathbf{P})$ be a probability space on which $(\Gamma_j)_{j\geq 0}$ are the arrival times of a unit rate Poisson process on $(0,\infty)$,  $(\zeta_j)_{j\geq 0}$ are i.i.d.\ uniform on $[0, 1/2], $  and  $(U_j)_{j\geq 0}$ are i.i.d.\ uniform on $[0, 1]$, all sequences being independent each other. We denote the random element $ (\Gamma_j,\zeta_j)_{j\geq 0}$ by $\omega$, and think of $\omega$ as a ``random environment". 
We will use $\mathbf{P}^\omega$  and  $\mathbf{E}^\omega$   to denote the conditional probability and the conditional expectation  given $\omega$. In other words, in $\mathbf{E}^\omega$, we take expectation with respect to the randomness of  $(U_j)_{j \ge 0}$ only.

Below we list some full-measure outcomes for the environment $\omega$, measurable with respect to the Borel $\sigma$-algebra on  $\R^{\Z_+}_+ \times \R^{\Z_+}_+ $, 
 which we will use throughout the paper.
\begin{align*}
 \Omega_1 &= \{ (x, y)  \in \R^{\Z_+}_+ \times \R^{\Z_+}_+ : \lim_{ j \to \infty}  x_j/ j \to 1   \},  \\
 \Omega_2 &=  \{ (x, y)  \in \R^{\Z_+}_+ \times \R^{\Z_+}_+ : \text{each $y_j$ is irrational} \} \\
 \Omega_3 &=  \{ (x, y)  \in \R^{\Z_+}_+ \times \R^{\Z_+}_+ : y_0, y_1, y_2,  \ldots \text{ are rationally independent} \} \subseteq  \Omega_2, \\
 \Omega_0 &=   \Omega_1 \cap \Omega_3.
\end{align*}
By the strong law of large numbers, $\mathbf{P}(\omega \in \Omega_1) = 1$. On the other hand, for any $k \ge 1$, the set of $k$ rationally dependent real numbers has zero Lebesgue measure on $\R^k$. Therefore, $\mathbf{P}(\omega \in \Omega_3)  = \mathbf{P}(\omega \in \Omega_2) = 1$, which implies that $ \mathbf{P}(\omega \in \Omega_0) = 1$. 

For $k \in \Z$, let 
\begin{equation}\label{def:rho_k}
   \varrho_k =    \varrho_k^\omega =  2 \sum_{j=0}^\infty \Gamma_j^{-1/\alpha}\cos (2\pi (U_j  + k\zeta_j )). 
\end{equation}
For a given $\omega$, we view $ \varrho_k^\omega$ as a function of $U = (U_j)_{j \ge 0}$. 
It is easy to check (see Proposition \ref{prop 1.1}(a) below)  that for a fixed $\omega \in \Omega_1$,  the series in \eqref{def:rho_k} is convergent $\mathbf{P}^\omega$-almost surely.
Define $\Lambda = \Lambda^\omega = \mathrm{diag}((\varrho_k^\omega)_{k \in \Z})$ to be the (random) diagonal operator on $\ell^2(\Z)$ and set  
\[ \Delta = \Delta^\omega =  \Pi \Lambda^\omega \Pi : \ell^2(\Z) \to \ell^2(\Z), \]
where the above multiplication should be understood as the composition of operators. 
 For $ \alpha \in [1, 2)$,  the operators $\Lambda$ and $\Delta$  become  unbounded  almost surely, so we need to be careful about defining their domains of definition. 
 
 Note that for each $l$, the vector $\Pi e_l$ decays like $|(\Pi e_l)_k | \le  O_l((1+k^2)^{-1})$ for all $k$. It then follows from Proposition \ref{prop 1.1}(a) that for any fixed  $\omega \in \Omega_1$, we have  $\Lambda^\omega  \Pi e_l  \in \ell^2(\Z)$ $\bP^\omega$-almost surely.  Therefore,  for each $\omega \in \Omega_1$,
 \[ \mathbf{P}^\omega ( \Delta^\omega e_l \in \ell^2(\Z) \text{  for all  } l \in \Z ) =1.\]
 Let $\mathcal{C}$ be the set of all finitely supported vectors in $\ell^2(\Z)$, which is dense in $\ell^2(\Z)$. 
 By taking finite linear combinations of the standard basis vectors,  we can now extend the definition of $\Delta^\omega$ to $\mathcal{C}$, $\bP^\omega$-almost surely. Thus, for each $\omega \in \Omega_1$, 
 the operator $\Delta^\omega$ is densely defined on $\ell^2(\Z)$ $\bP^\omega$-almost surely.  Clearly, the operator $\Delta^\omega$ is Hermitian. In Proposition \ref{prop 1.1}(b), we will show that it is also self-adjoint $\bP^\omega$-almost surely.
 
\begin{prop} \label{prop 1.1}
For each $\omega \in \Omega_1$, the following statements hold $\bP^\omega$-almost surely. 
\begin{enumerate}
\item[(a)] For each $k \in \Z$, $\varrho^\omega_k $ is finite. Moreover,  $\sum_{k\in \Z} (1+k^2)^{-1} ( \varrho^\omega_k)^2 <\infty$.
\item[(b)] $\Delta^\omega$ is a self-adjoint operator on the domain $\D^\omega = \{ v \in \lz : \|\Delta^\omega v\|_2 < \infty\}$.
\end{enumerate}
\end{prop}

Let $\mathcal{M}$ be the set of probability measures on $\R$. Define the L\'evy distance between two probability measures $\nu_1, \nu_2 \in \mathcal{M}$ as \begin{align*}
\dl(\nu_1, \nu_2) = \inf \{ \varepsilon >0 : \nu_1((-\infty, t-\varepsilon]) - \varepsilon \leq \nu_2((-\infty,t]) \leq \nu_1((-\infty, t+\varepsilon]) + \varepsilon \mathrm{\;for\;all\;} t\in \R\}.
\end{align*}
It is well known that $0 \le  \dl(\nu_1, \nu_2) \le 1$ and the space $\mathcal{M}$ equipped with the L\'evy distance is a complete separable metric space. 

A random probability measure on a measurable space $(\Sigma, \mathcal{S})$ is a measurable map $\mu: \Sigma \to \mathcal{M}$ with $\sigma \in \Sigma \mapsto \mu^\sigma \in \mathcal{M}$. We say that a sequence of random probability measures $(\mu_n)_{n\geq 1}$ on $\Xi$ converges weakly to the random probability measure $\mu_\infty$ on $\Upsilon$, and write $\mu_n \Rrightarrow \mu_\infty$, if for all bounded and continuous functions $\phi: \mathcal{M}\to \R$, \begin{align*}
\mathsf{E} \phi(\mu_n) \to \bE \phi(\mu_\infty)
\end{align*}
as $n \to \infty$, where $\mathsf{E}$ and $\bE$ denote the expectations with respect to the probability measures $\sP$ and $\bP$ respectively. We will use the notation $\Rightarrow$ to denote weak convergence of random variables or their laws and use $\ed$ to denote equality in distribution.

For  a self-adjoint operator $\Phi$  on  $\lz$ and a unit vector $v \in \lz$, we denote by $\nu_{\Phi,v}$ the spectral measure of the operator at $v$, i.e., $\nu_{\Phi,v}$ is the unique probability measure on $\R$ that satisfies \begin{align*}
\langle v, f(\Phi) v\rangle = \int_{\R}fd\nu_{\Phi,v},
\end{align*}
for any bounded measurable function $f:\R \to \R$. Alternatively, the probability measure $\nu_{\Phi,v}$ is described by its  Stieltjes transform:
\[ \int \frac{1}{x -z}\;d\nu_{\Phi,v}(x) = \langle v, (\Phi - z)^{-1} v\rangle, \qquad z \in \C\setminus \R.   \]
When $v = e_0$, we shall denote $\nu_{\Phi, v}$ simply by $\nu_\Phi$. Let us define
\[ \nu_\ct = \nu_\ct^\omega = \bE^\omega \nu_{\Delta, u},\]
 where the unit vector $u =\sqrt{2}  \Pi e_0$. Note that $\nu_\ct$ is  a random probability measure on $\Upsilon$. We are now ready to state our main result.

\begin{theorem} \label{main_thm} Fix $\alpha \in (0, 2)$. 
Let $(T_N)_{N \ge 1}$ be a sequence of heavy-tailed symmetric Toeplitz matrices whose entry distribution satisfies \eqref{eqn 1.1} and \eqref{eqn 1.2} and let $(\mu_{T_N})_{N \ge 1}$ be their empirical spectral distributions.  Then, as $N \to \infty$, 
\[ \mu_{T_N} \Rrightarrow \nu_\ct. \]
\end{theorem}

We list some properties of $\nu_\ct$ in the next theorem.

\begin{theorem} \label{thm 1.3}
The following statements hold for each $\omega \in \Omega_0$.
\begin{enumerate}
\item[(a)]  $\nu_\ct^\omega$ is a probability distribution symmetric around 0, 
\item[(b)]  $\nu_\ct^\omega$ is subgaussian. In particular,  for every $\beta>0$, 
\begin{align*}
\int_\R e^{\beta t} \nu_\ct^\omega(dt) \leq 2 \exp\Bigl(2\beta^2\sum_{j=0}^\infty \Gamma_j^{-2/\alpha} \Bigr).
\end{align*}
\item[(c)]  The support of $\nu_\ct^\omega$ is contained inside the bounded interval $$\Big[-2\sum_{j=0}^\infty \Gamma_j^{-1/\alpha}, 2\sum_{j=0}^\infty \Gamma_j^{-1/\alpha}\Big]$$
when $0<\alpha<1$.
\item[(d)]  $\nu_\ct^\omega$ has an unbounded support when $1\leq \alpha <2$.
\end{enumerate}
\end{theorem}
It is an interesting open problem to show that $\nu_\ct^\omega$ is absolutely continuous $\bP^\omega$-a.s.. One of the main difficulties in showing the absolute continuity lies in the fact that although each fixed $\omega$,
$\varrho^\omega_k$ is absolutely continuous (with respect to the randomness of $(U_j)_{j \ge 0}$), they are not independent.  However, we would like to mention that for fixed $\omega \in \Omega_0$,  the random operator $\Delta^\omega$, viewed as a function of $(U_j)_{j \ge 0}$,  is ergodic, as explained in the next subsection. We use this  fact in showing part (d) of Theorem~\ref{thm 1.3}. We  believe that it might be helpful in showing the absolute continuity of the measure  
$\nu_\ct^\omega$  along the lines of argument in \cite{Del}. 

\subsection{Ergodicity of the operator \texorpdfstring{$\Delta^\omega$}{Delta}}
Throughout this subsection, we fix $\omega \in \Omega_0$.
Without loss, let us assume that the random variables $U = (U_j)_{j \ge 0}$ are the coordinate-wise projections from   $[0, 1)^{\Z_+}$ equipped with the Borel $\sigma$-algebra and the product Lebesgue measure. 
Define the coordinate-wise rotation map $\theta^\omega$ on $[0, 1)^{\Z_+}$
\[ \theta^\omega( U )  =  ( \{ U_j+  \zeta_j\} )_{j \ge 0},\]
where for a positive real $x$, $\{ x\}$ denotes its fractional part. 
Clearly, $\theta^\omega$ is measure preserving and since $\zeta_j$'s are irrational, $\theta^\omega$ is also ergodic. Moreover, $\theta^\omega$ acts ergodically on the sequence $(\varrho_k^\omega )_{ k \in \Z}$ as 
 \begin{equation}\label{eq:ergodic}
 \Big(\varrho_k^\omega \big(  (\theta^\omega)^l (U) \big ) \Big)_{k \in \Z}  =    \big(\varrho^\omega_{k+l}\big)_{k \in \Z}, \quad \text{ for any } l \in \Z. 
  \end{equation}
This action can be naturally extended to the random operator $\Delta^\omega$ as 
\[  \Delta^\omega \big( (\theta^\omega)^l  (U) \big) = \Pi\; \mathrm{diag} \big(\varrho_k^\omega \big(  (\theta^\omega)^l (U) \big ) \big)_{k \in \Z}  \Pi.\]
This implies that the self-adjoint random operator $\Delta^\omega$ is ergodic, i.e.,  for each $l \in \Z$, there exists a unitary operator $\mathfrak{U}_l = \mathfrak{U}^\omega_l$ on $\ell^2(\Z)$ such that
 \begin{equation}\label{eq:ergodic_op1}
\Delta^\omega ( (\theta^\omega)^l  (U) ) = \mathfrak{U}^*_l \Delta^\omega(U) \mathfrak{U}_l.
  \end{equation}
  See \cite{Aiz} or \cite{Kir} for more information on ergodic operators. 
Indeed, take $\mathfrak{U}_l$ to be the right $l$-shift operator on $\ell^2(\Z)$, i.e.,  $ (\mathfrak{U}_l v)_n =  v_{n -l}$ for all $n \in \Z$. Since $\Pi$ is translation invariant, it commutes with $\mathfrak{U}_l$. This, coupled with \eqref{eq:ergodic}, implies \eqref{eq:ergodic_op1}. We record this observation in the following lemma.

\begin{lemma}\label{lem:ergodic}
For each $\omega \in \Omega_0$, the operator $\Delta^\omega$ is ergodic with respect to the rotation $\theta^\omega$. 
Consequently, there exists a closed set $\Sigma^\omega \subseteq \R,$ 
such that 
\begin{align*}
\sigma( \Delta^\omega ) = \Sigma^\omega \quad \mathbf{P}^\omega\text{-a.s.},  
\end{align*}
where $\sigma( \Delta^\omega )$ is the spectrum of the operator $\Delta^\omega $. The same applies to the absolutely continuous, singular continuous, and pure point spectrum of  $\Delta^\omega$. 
\end{lemma}
The consequence mentioned in the above lemma is due to  Pastur's theorem (see, for example, \cite[Theorem 3.10]{Aiz}),

In the rest of the paper,  assuming that the underlying probability space is clear from the context, we shall use $\P, \E$ to denote the probability and expectation over all random variables involved in the expression under consideration. Also, we will use $\P^\omega$ and  $\E^\omega$ to denote the conditional probability and expectation given $\omega$.

\section{Roadmap for the proof of Theorem \ref{main_thm}}

\subsection{Connection between Toeplitz and circulant matrices} \label{sec 2.1}

Following \cite{Sen2}, we observe that the Toeplitz matrix $ T_N$ is the principal submatrix of the $2N \times 2N$ circulant matrix

\begin{align}\label{eq:matrix G}
G_{2N} = \begin{pmatrix}
b_0 & b_1 & \ldots & b_{N-1} &  b_N & b_{N-1} & \ldots & b_2 & b_1\\
b_1 & b_0 & b_1 & \ddots & b_{N-1} &  b_N & b_{N-1} & \ddots & b_2\\
\vdots & \ddots & \ddots & \ddots & \ddots & \ddots & \ddots & \ddots & \vdots\\
b_2 & \ddots & \ddots & \ddots & \ddots & \ddots & \ddots & \ddots & b_1\\
b_1 & b_2 & \ldots & \ldots & \ldots & \ldots & \ldots & b_1 & b_0
\end{pmatrix} = (b_{\min(|k-l|,2N-|k-l|)})_{k,l\in[2N]}.
\end{align}
 In other words, letting \begin{align*}
Q_{2N} = \begin{pmatrix}
I_N	& 0_N\\
0_N & 0_N
\end{pmatrix},
\end{align*}
we have 
\[ \begin{pmatrix}
T_N	& 0_N\\
0_N & 0_N
\end{pmatrix} =  Q_{2N} G_{2N}Q_{2N}. \]
The choice of $b_N$ does not affect the above observation and we set $b_N = 0$.
Working with circulant matrices has the advantage that they can be diagonalized by the discrete Fourier transform matrix. More precisely, let $F_{2N}$ denote the $2N \times 2N$ discrete Fourier transform matrix,
\begin{align*}
F_{2N} = \Bigl(\frac{1}{\sqrt{2N}}\exp \Bigl(\frac{2\pi ikl}{2N}\Bigr)\Bigr)_{k,l\in [2N]}.
\end{align*} 
Then we can write $G_{2N} = F_{2N} D_{2N}^\circ F^*_{2N}$, where
$D^\circ_{2N} = \diag(d^\circ_0, \ldots, d^\circ_{2N-1})$ is the diagonal matrix of eigenvalues of $G_{2N}$ which are given by
\begin{align*}
d_k^\circ = \sum_{j=0}^{2N-1} b_{\min(j, 2N - j)} \exp\Bigl(\frac{2\pi ijk}{2N}\Bigr) = b_0 + 2\sum_{j=1}^{N-1} b_j \cos \Bigl(\frac{2\pi jk}{2N}\Bigr), \qquad k\in [2N].
\end{align*}
The matrix $Q_{2N} G_{2N}Q_{2N}$ has the same eigenvalues as its discrete Fourier conjugate \begin{align*}
F_{2N}^*Q_{2N} G_{2N}Q_{2N}F_{2N} = F_{2N}^*Q_{2N} F_{2N} D_{2N}^\circ F_{2N}^*Q_{2N}F_{2N} = P_{2N}D_{2N}^\circ P_{2N}
\end{align*}
where the matrix $P_{2N}:= F_{2N}^*Q_{2N} F_{2N}$ is a projection matrix with  entries \begin{align*}
P_{2N}(k,l) = \begin{cases}
\frac{1}{2}; & \mathrm{if\;} k = l,\\
0; & \mathrm{if\;} k \neq l\mathrm{\;and\;} |k-l| \mathrm{\;is\;even},\\
\frac{1}{N}\Bigl(1-\exp\Bigl(-\frac{2\pi i (k-l)}{2N}\Bigr)\Bigr)^{-1}; & \mathrm{if\;}|k-l| \mathrm{\;is\;odd}.
\end{cases}
\end{align*}
From the above observations we have
\begin{align} \label{eq:esd_toep_circ}
\frac{1}{2}(\mu_{ T_N} + \delta_0) = \mu_{Q_{2N}  G_{2N} Q_{2N}} = \mu_{P_{2N} D_{2N}^\circ P_{2N}}.
\end{align}
We will approximate $d_k^\circ$ by 
\begin{align*}
 d_k = 2\sum_{j=0}^{N-1}  b_j \cos \Bigl(\frac{2\pi jk}{2N}\Bigr),
\end{align*}
and define $ D_{2N} = \diag(d_0, \ldots,  d_{2N-1})$.

\subsection{Truncation of matrices and operators}
We will establish Theorem~\ref{main_thm} via the method of moments. To execute it, we first need to perform several truncations on the variables involved. The level of these truncations would be  measured  by positive integers $K$, $M$, and $L$.

First we truncate the variable $b_j$ by setting  $ b_j^\M:= \sgn(b_j)\min(|b_j|,M)$, $j \in [N]$. Replacing $b_j$ by $b_j^\M$, we obtain truncated versions of $d_k, D_{2N}$, and $G_{2N}$, which we will denote by $d^\M_k, D^\M_{2N}$, and $G^\M_{2N}$ respectively. 

Let $|b_{(0)}|\geq  \ldots\geq  |b_{(N-1)}|$ be the order statistics of $|b_0|, \ldots, |b_{N-1}|$ and let $\sigma_N:[N] \to [N]$ be the uniform random permutation such that $ b_{(j)} =  b_{\sigma_N(j)}$ for $j \in [N]$. If we now define $b_{(j)}^\M~= \sgn(b_{(k)})\min(|b_{(j)}|,M), j \in [N]$, then we can write \begin{align*}
 d_k^\M = 2\sum_{j=0}^{N-1} b_{(j)}^\M\cos \Bigl(\frac{2\pi k\sigma_N(j)}{2N} \Bigr).
\end{align*}
Next, we keep the first $K$ many terms in the above term and define
\begin{align*}
 d_k^\MK = 2\sum_{j=0}^{K-1}  b_{(j)}^\M \cos \Bigl(\frac{2\pi k\sigma_N(j)}{2N}\Bigr).
\end{align*}
The corresponding diagonal matrix is denoted by  $D_{2N}^\MK = \diag(d_0^\MK, \ldots,  d_{2N-1}^\MK)$.

In our last truncation, for a given band width $L$, we truncate the matrix $P_{2N}$ in a circular fashion by setting
\begin{align*}
P^\L_{2N}(k,l) = \begin{cases}
P_{2N}(k,l); & |k-l| \leq L \mathrm{\;or\;} |k-l| \geq  2N-L,\\
0; & \mathrm{otherwise}.
\end{cases}
\end{align*}

To keep the notation light, from now on we shall drop the subscripts mentioning the matrix size.

We perform analogous truncations on the operator $\Delta$ as well. 
Let $\Pi^\L:\lz\to \lz$ be obtained from $\Pi$ by zeroing out the entries of $\Pi$ outside the band of width $L$ around its diagonal. That is, 
\begin{align*}
\Pi^\L(k,l): = \langle e_k, \Pi^\L e_l\rangle = \begin{cases}
\Pi(k,l), & \text{if } |k-l| \leq L,\\
0 & \mathrm{otherwise}.
\end{cases}
\end{align*}
$\Pi^\L$ still remains a bounded operator on $\lz$ (see Lemma \ref{lem pi_L bounded}). 

Let  $\Gamma_j^\M = \max(\Gamma_j, M^{-\alpha})$ and  define
\begin{align*}
\varrho_k^\M  = 2\sum_{j=0}^\infty (\Gamma_j^\M)^{-1/\alpha}\cos (2\pi (U_j + k\zeta_j)), \qquad
\varrho^\MK_k  = 2\sum_{j=0}^{K-1} (\Gamma_j^\M)^{-1/\alpha}\cos (2\pi (U_j + k\zeta_j)).
\end{align*}
Denote the  corresponding diagonal operators as $ \Lambda^\M = \mathrm{diag}((\varrho_k^\M)_{k \in \Z})$ and $ \Lambda^\MK = \mathrm{diag}((\varrho_k^\MK)_{k \in \Z})$.
Note that $|\varrho_k^\MK| \le MK$ for each $k$ and consequently, $\Lambda^\MK$ is a bounded operator with $\|\Lambda^\MK\|_\op \le MK$ and hence it is self-adjoint.

We shall use the following notation in the rest of the paper. Let $H$ denote the matrix $PDP$. We use $\fH$, $\sH$ and $\tH$ to respectively denote the matrices $PD^\M P$, $P^\L D^\M P^\L$ and $P^\L D^\MK P^\L$. Likewise, we define $\fD = \Pi\Lambda^\M\Pi$, $\sD = \Pi^\L \Lambda^\M \Pi^\L$ and $\tD = \Pi^\L \Lambda^\MK \Pi^\L$.

\subsection{Proof of Theorem \ref{main_thm}}
We identify the three essential steps that go into the proof of our main result. The standing assumption throughout the paper is the following relationship among the different truncation levels:
\begin{align}\label{eq:trunc_assump}
    K = M = L^{1/9}.
\end{align}
 In the first step, using the method of moments, we show that the empirical spectral distribution of the truncated matrix $\tH$ is close to the $U$-averaged spectral measure of the truncated operator $\tD$ at $e_0$.

\begin{prop} \label{prop 3.1}
Fix $K, L, M \ge 1$.
There exist random probability measures $ (\mu_N^\MLK)_{N \ge 1}$ and $ \vartheta^\MLK$  defined on a common probability space such that $\mu_N^\MLK \ed \mu_{\tH}$, $\vartheta^\MLK \ed \E^\omega \nu_{\tD}$ and almost surely, as $N\to\infty$,
\begin{align*}
\dl\lp \mu_N^\MLK, \vartheta^\MLK  \rp \to 0.
\end{align*}
\end{prop}

The second step comprises showing that the empirical spectral distributions of $H$ and $\tH$ are close.

\begin{prop} \label{prop 4.2}
Assume \eqref{eq:trunc_assump}. Then
\begin{align*}
\lim_{L \to \infty} \limsup_N \E \dl(\mu_H, \mu_{\tH}) = 0.
\end{align*}
\end{prop}

In the final step, we show that the $U$-averaged spectral measure of $\Delta$ at $e_0$ is well approximated by that corresponding to $\tD$.

\begin{prop}\label{prop 5.1}
Assume \eqref{eq:trunc_assump}. Then
\begin{align*}
\lim_{L\to \infty} \E \dl(\E^\omega \nu_\Delta, \E^\omega \nu_{\tD}) = 0.
\end{align*}
\end{prop}

The results in the three propositions above are tied together by the following lemma to establish that $\mu_H \Rrightarrow \E^\omega\nu_\Delta$.

\begin{lemma}\label{lem 2.4}
For fixed $K, L, M \ge 1$, consider two sequences of random probability measures $ (\mu_N^\MLK)_{N \ge 1}$ and $ \vartheta^\MLK$ defined on a common probability space such that $\mu_N^\MLK \ed \mu_{\tH}$, and $\vartheta^\MLK \ed \E^\omega \nu_{\tD}$. Assume that $K = K(L)$ and $M = M(L)$ are functions of $L$ such that $K(L), M(L) \to \infty$ as $L \to \infty$. Suppose further that the following are satisfied:
\begin{enumerate}
\item $\lim_{L \to\infty} \limsup_N \E \dl(\mu_{\tH},\mu_H) = 0$,
\item $\lim_{L\to \infty} \E \dl(\E^\omega\nu_{\tD}, \E^\omega\nu_\Delta) = 0$, and 
\item for each $L\geq 1$, $ \lim_{N\to\infty} \dl\lp \mu_N^\MLK, \vartheta^\MLK  \rp = 0$ almost surely. 
\end{enumerate}
Then, $\mu_H \Rrightarrow \E^\omega \nu_\Delta,$ as $N \to \infty$.
\end{lemma} 
\begin{proof}
By the portmanteau theorem (\cite[Theorem 2.1]{Bil2}), it suffices to show that for any uniformly continuous and bounded function $\phi:\mathcal{M} \to \R$, $\E \phi(\mu_H) \to \E \phi(\E^\omega\nu_\Delta)$ as $N\to \infty$. To that extent we have
\begin{align}
    \lim_{N\to\infty}|\E \phi(\mu_H) - \E \phi(\E^\omega\nu_\Delta)| & \leq \lim_{L\to \infty}\limsup_N|\E \phi(\mu_H) - \E \phi(\mu_{\tH})| \label{eq:lem 2.4.0}\\
    &  \qquad +\lim_{L\to \infty}\lim_{N\to \infty}|\E \phi(\mu_N^{\tH}) - \E \phi(\vartheta^{\MLK})|  \label{eq:lem 2.4.1}\\
    & \qquad + \lim_{L\to \infty}| \E \phi(\E^\omega\nu_{\tD}) - \E \phi(\E^\omega\nu_\Delta)| \label{eq:lem 2.4.2} 
\end{align}
The term in \eqref{eq:lem 2.4.2} vanishes because of the second assumption. The term in \eqref{eq:lem 2.4.1} vanishes by the dominated convergence theorem and the third assumption in the lemma. We now show that the right-hand side of \eqref{eq:lem 2.4.0} vanishes as well. Fix $\varepsilon > 0$. By the uniform continuity of $\phi$, there exists a $\delta >0$ such that for $\nu_1$, $\nu_2 \in \mathcal{M}$ satisfying $\dl(\nu_1,\nu_2) \leq \delta$, $|\phi(\nu_1) - \phi(\nu_2)| \leq \varepsilon$. Then \begin{align*}
	|\E \phi(\mu_H) - \E \phi(\mu_{\tH})| & \leq \E |\phi(\mu_H) - \phi(\mu_{\tH})|\1_{\dl(\mu_H, \mu_{\tH}) \leq \delta} \\
    & \qquad + \E |\phi(\mu_H) - \phi(\mu_{\tH})|\1_{\dl(\mu_H, \mu_{\tH}) > \delta}\\
	& \leq \varepsilon \P (\dl(\mu_{\tH}, \mu_H) \leq \delta) + 2\|\phi\|_\infty \mathbb{P} (\dl(\mu_{\tH}, \mu_H) > \delta) \\
	& \leq \varepsilon + \frac{2\|\phi\|_\infty}{\delta}\E \dl (\mu_{\tH}, \mu_H). 
\end{align*}
Since $\varepsilon$ is arbitrary, the result follows from the first assumption of the lemma.
\end{proof}

Recall identity \eqref{eq:esd_toep_circ} that relates the empirical spectral distributions of $T$ and $PD^\circ P$. Once we show that the replacement of $D^\circ$ by $D = D^\circ + b_0 I$ does not affect the limiting distribution (see Lemma \ref{lem 4.2}) we can conclude that $\mu_T \Rrightarrow 2\E^\omega \nu_\Delta - \delta_0$. Thus, it remains to show that 
\begin{align}\label{eq: avg_sp_measure}
    2\E^\omega \nu_\Delta - \delta_0 = \E^\omega \nu_{\Delta, u}.
\end{align}
In Lemma \ref{lem spec_measure_at_u} we prove that almost surely, $2\nu_\Delta - \delta_0 = \nu_{\Delta, u}$. Since the random element $\omega$ is independent of $U$, for any bounded measurable function $f:\R \to \R$, we have \begin{align*}
    \int f d(2\E^\omega \nu_\Delta - \delta_0) = \E^\omega \int f d(2\nu_\Delta - \delta_0) = \E^\omega \int f d(\nu_{\Delta,u}) = \int f d(\E^\omega \nu_{\Delta,u}),
\end{align*}
establishing \eqref{eq: avg_sp_measure} and completing the proof of Theorem \ref{main_thm}.

\subsection{Organization of the paper}
The rest of the paper is organized as follows. In Sections  \ref{sec 3} and \ref{sec 4} we prove Propositions \ref{prop 3.1} and \ref{prop 4.2} respectively. While Propositions \ref{prop 1.1} and \ref{prop 5.1} are proved in Section \ref{sec 5}, the proof of Theorem \ref{thm 1.3} is the content of Section \ref{sec 6}.

\section{Convergence of the truncated spectral measures}\label{sec 3}

In this section, we shall prove Proposition \ref{prop 3.1}. For the rest of this section, we shall assume that $K$, $L$ and $M$ are fixed positive integers. Recall that we need to show that there exist random probability measures $ (\mu_N^\MLK)_{N \ge 1}$ and $ \vartheta^\MLK$   on a common probability space such that $\mu_N^\MLK \ed \mu_{\tH}$, $\vartheta^\MLK \ed \E^\omega \nu_{\tD}$ and almost surely, $\dl\lp \mu_N^\MLK, \vartheta^\MLK  \rp \to 0$ as $N\to\infty$.

At the heart of our proof lies Weyl's equidistribution criterion for measures on the multi-dimensional torus, which we state below and include a proof of for completeness. 
\begin{lemma}[Weyl's equidistribution criterion]\label{lem a1}
Let $(\mu_n)_{n\geq 1}$ be probability measures on the $m$-dimensional torus $\mathbb{T}^m = [0,1]^m$. Then, $(\mu_n)_{n\geq 1}$ converges weakly to the uniform distribution on $\mathbb{T}^m$ if and only if for every $h \in \Z^m\setminus \{0\}$, \begin{align}\label{lem a1 eqn 1}
\int_{\mathbb{T}^m} e^{2\pi i \langle h,x\rangle} d\mu_n(x) \to 0.
\end{align}
\end{lemma}

\begin{proof}
Let $\mu_0$ be the uniform distribution on $\mathbb{T}^m$. By the portmanteau theorem, $\mu_n \Rightarrow \mu_0$ if and only if for every bounded continuous complex-valued function $f$ on $\mathbb{T}^m$, \begin{align}\label{lem a1 eqn 2}
\int_{\mathbb{T}^m} f d\mu_n \to \int_{\mathbb{T}^m} f d\mu_0.
\end{align}
Since $\mathbb{T}^m$ is compact, it suffices to consider continuous complex-valued $f$ in the above display as they are automatically bounded. Choosing $f(x) = e^{2\pi i \langle h,x\rangle}$, for every $h \in \Z^m\setminus \{0\}$, we have 
\begin{align*}
\int_{\mathbb{T}^m} e^{2\pi i \langle h,x\rangle} d\mu_n(x) & \to \int_{\mathbb{T}^m} e^{2\pi i \langle h,x\rangle} d\mu_0(x) = \int_{[0,1]^m} e^{2\pi i \langle h,x\rangle} dx = 0.
\end{align*} 
This establishes the necessity of $\mu_n \Rightarrow \mu_0$ to guarantee \eqref{lem a1 eqn 1}. To show that \eqref{lem a1 eqn 1} is sufficient, we first note that \eqref{lem a1 eqn 1} implies that
\[\int_{\mathbb{T}^m} e^{2\pi i \langle h,x\rangle} d\mu_n(x) \to \int_{\mathbb{T}^m} e^{2\pi i \langle h,x\rangle} d\mu_n(x), \]
for every $h \in \Z^m$, including $h=0$ for which the convergence trivially holds since both sides are equal to $1$. Now by the Weierstrass approximation theorem that any continuous complex-valued function on $\mathbb{T}^m$ can be approximated arbitrarily closely in the uniform norm by a finite linear combination of the functions $(e^{2\pi i \langle h,\cdot\rangle})_{h \in \Z^m}$, \eqref{lem a1 eqn 2} follows.
\end{proof}

As an application of Weyl's equidistribution criterion, we prove the following.
\begin{lemma} \label{lem 3.2}
Let $\tau_N : [N]\to [N]$ be a deterministic sequence permutations such that as $N\to \infty$,
\begin{align} \label{lem 3.2 eqn 1}
\Bigl(\frac{\tau_N(0)}{2N}, \ldots, \frac{\tau_N(K-1)}{2N}\Bigr) \to (\xi_0, \ldots, \xi_{K-1})
\end{align}
 for some $\xi_0, \ldots, \xi_{K-1} \in [0,1/2]$ linearly independent over $\Z$.  Let $\Theta_N$ be a random variable uniform on $[2N]$. Then, \begin{align*}
\Bigl(\bb{\frac{\Theta_N\tau_N(0)}{2N}}, \ldots, \bb{\frac{\Theta_N\tau_N(0)}{2N}}\Bigr) \Rightarrow (V_0, \ldots, V_{K-1})
\end{align*}
where $(V_j)_{j\in[K]}$ are i.i.d.\  random variables uniform on $[0, 1]$.
\end{lemma}

\begin{proof}
By Lemma \ref{lem a1}, it suffices to show that for $h \in \Z^K\setminus \{0\}$,  \begin{align*}
\frac{1}{2N}\sum_{j = 0}^{2N-1} \exp\Bigl(2\pi i \sum_{k = 0}^{K-1} h_k\bb{\frac{j\tau_N(k)}{2N}}\Bigr) \to 0 
\end{align*}
 as $N \to \infty$. Indeed, \begin{align*} 
	\frac{1}{2N}\sum_{j = 0}^{2N-1} \exp\Bigl(2\pi i \sum_{k = 0}^{K-1} h_k\bb{\frac{j\tau_N(k)}{2N}}\Bigr) & = \frac{1}{2N}\sum_{j = 0}^{2N-1} \exp\Bigl(2\pi i \sum_{k = 0}^{K-1} h_k\frac{j\tau_N(k)}{2N}\Bigr) \\
	& = \frac{1}{2N}\frac{1 - \exp(2\pi i \sum_{k=0}^{K-1}h_k\tau_N(k))}{1 - \exp(2\pi i \sum_{k=0}^{K-1}\frac{h_k\tau_N(k)}{2N})}.
\end{align*}
By \eqref{lem 3.2 eqn 1}, $(2N)^{-1}\sum_{k=0}^{K-1}h_k\tau_N(k) \to \sum_{k=0}^{K-1} h_k \xi_k,$ which is non-zero since $\xi_k$'s are linearly independent over $\Z$. On the other hand, 
we can trivially upper bound $1 - \exp(2\pi i \sum_{k=0}^{K-1}h_k\tau_N(k))$ in absolute value by~$2$.
Therefore, the last line of the display above approaches 0 as $N\to \infty$.
\end{proof}

Recall the definition of $p$ from \eqref{eqn 1.2}. On  $(\Upsilon, \mathcal{G},\mathbf{P})$,  let $(\varepsilon_j)_{j\geq 0}$ be i.i.d.\ Rademacher random variables  satisfying \begin{align*}
\P(\varepsilon_0 = 1) = p = 1 - \P(\varepsilon_0 = -1)
\end{align*}
which are independent of all other randomness. Recall that $\sigma_N:[N] \to [N]$ is the uniform random permutation such that $|b_{\sigma_N(0)}|\geq  \ldots\geq  |b_{\sigma_N(N-1)}|$. The result in the next lemma is standard in the theory of heavy-tailed random variables (see, for example, \cite[Lemma 1]{Kni}).

\begin{lemma}\label{lem 3.3}
As $N\to \infty$, 
\begin{align*}
	 \Bigl( |b_{(j)}|, \sgn(b_{(j)}), \frac{\sigma_N(j)}{2N}\Bigr)_{j\in[K]}
	& \Rightarrow \Bigl(\Gamma_j^{-1/\alpha}, \varepsilon_j,\zeta_j\Bigr)_{j\in[K]}.
\end{align*}
\end{lemma}

As a consequence, we have

\begin{lemma}\label{lem 3.4}
Let $\Theta_N$ be a random variable in $(\Xi, \mathcal{F},\mathsf{P})$ that is uniform on $[2N]$ and independent of $(a_j)_{j \ge 0}$. Then, as $N\to \infty$,
\begin{align*}
	\Bigl( |b_{(j)}|, \sgn(b_{(j)}), \frac{\sigma_N(j)}{2N}, \bb{\frac{\Theta_N\sigma_N(j)}{2N}}\Bigr)_{j\in[K]}
	& \Rightarrow \Bigl(\Gamma_j^{-1/\alpha}, \varepsilon_j, \zeta_j, U_j\Bigr)_{j\in[K]}
\end{align*}
and 
\begin{align*}
	\Bigl( b_{(j)}, \frac{\sigma_N(j)}{2N}, \bb{\frac{\Theta_N\sigma_N(j)}{2N}}\Bigr)_{j\in[K]}
	& \Rightarrow \Bigl(\varepsilon_j\Gamma_j^{-1/\alpha}, \zeta_j, U_j\Bigr)_{j\in[K]}.
\end{align*}
\end{lemma}

\begin{proof}
The second statement of the lemma immediately follows from the first  using the continuous mapping theorem. So, let us prove the first statement.

By the Skorokhod representation theorem, we may switch to a probability space where the convergence in Lemma \ref{lem 3.3} happens almost surely. Since the distribution of $(\zeta_0, \ldots, \zeta_{K-1})$ is absolutely continuous, $\zeta_0,\ldots, \zeta_{K-1}$ are linearly independent over $\Z$ with probability 1. We fix a realization so that we have the pointwise convergence of $X_N:=\big( |b_{(j)}|, \sgn(b_{(j)}), \tfrac{\sigma_N(j)}{2N}\big)_{j\in[K]}$ to $X:=\big(\Gamma_j^{-1/\alpha}, \varepsilon_j,\zeta_j\big)_{j\in[K]}$
  where $(\zeta_j)_{j \in [K]}$ are linearly independent over $\Z$. Given that, by Lemma \ref{lem 3.2}, the conditional distribution of $Y_N := \big(\big \{ \tfrac{\Theta_N\sigma_N(j)}{2N} \big\}\big)_{j\in [K]}$ converges in distribution to $Y:=(U_j)_{j \in [K]}$ and the limiting distribution is independent of the value of $X$. This shows the joint weak convergence of $(X_N, Y_N)$ to $(X, Y)$ and also that $X$ and $Y$ are independent of each other. 
  
\end{proof}

We need the following elementary estimate for the difference between the entries of $P$ and $\Pi$.

\begin{lemma} \label{lem 2.2}
We have
\begin{itemize}
\item[(a)] $|P(k, l)| \le 1$ for each $k, l \in [2N]$. Also, $|\Pi(k, l)| \le 1/2$ for $k, l \in \Z$.
\item[(b)]  For  $L \ge 1$, there exists a constant $C(L) $ such that  $|P(k,l) - \Pi(k,l)| \le  C(L) N^{-1}$ for any $k, l \in [2N]$ with $|k-l| \leq L$.
\end{itemize}
\end{lemma}
\begin{proof}
(a) The estimate on the entries of $\Pi$ follows directly from its definition \eqref{eqn 1.3}. Note that $P$ is a Hermitian projection matrix. Hence, $P = P^2 = P P^*$, which yields
$P(k, k) = \sum_{l \in [2N]} |P(k, l)|^2$.
Therefore, $|P(k,l)|^2 \le P(k, k) = 1/2$.

(b) See \cite[(15)]{Sen2} for a proof.
\end{proof}

\subsection{Proof of Proposition \ref{prop 3.1}}

We first show that for each sample realization, the probability measure $\E^\omega\nu_{\tD}$ is determined by its moments. For $r \in \N$, let $$m(r) := \int_{\R} t^r \E^\omega\nu_{\tD}(dt) = \E^\omega \int_{\R} t^r \nu_{\tD}(dt).$$ 
Then, we have 
\begin{align}
\big|m(r)\big| & =\big|\E^\omega(\tD)^r(0,0)\big| \leq \E^\omega(|\Pi^\L||\Lambda^\MK||\Pi^\L|)^r(0,0) \notag\\
& \leq (MK)^r \sum_{k_1, \ldots, k_{2r-1}\in \Z} \big|\Pi^\L(0, k_1)\big|\big|\Pi^\L(k_1, k_2)\big|\cdots \big|\Pi^\L (k_{2r-1},0)\big| \label{prop 3.1 eqn 1}
\end{align}
where we used that $| \varrho^{\MK}_k| \leq MK$. Note that each path having a non-zero contribution to the sum above is of the form $0=k_0\to k_1\to k_2 \to \cdots \to k_{2r-1} \to k_{2r}$ = 0 such that $|k_l-k_{l+1}| \leq L$ for all $l \in[2r]$. The number of such paths is at most $(2L+1)^{2r-1}$ and each of these paths contributes at most $2^{-(2r-1)}$ by virtue of the entries of $\Pi^\L$ being absolutely bounded above by $1/2$ (Lemma \ref{lem 2.2}(a)). Therefore, we deduce that for $r \ge 1$,
\[ |m(r)| \le (MK)^r 2^{-2r}\big(2L + 1\big)^{2r-1}. \]
Hence, by \cite[Theorem 30.1]{Bil1},  we have that $\E^\omega\nu_{\tD}$ is uniquely determined by  $(m(r))_{r\in \N}$.

By Skorokhod's representation, we may consider the random variables to be defined on a common probability space so that 
\begin{equation}\label{eq:pt_conv_assumption}
\text{the convergence in the first statement of Lemma \ref{lem 3.4} happens almost surely.}
\end{equation}
Note that the truncated matrix $\tH$ and the truncated operator $\tD$ are functions of  the random variables $(b_{(j)},\sigma_N(j))_{j\in[K]}$ and $(\Gamma_j,\zeta_j,U_j)_{j \in [K]}$ respectively. This allows us to define, on the new common probability space, 
random probability measures  $(\mu_N^{\MLK})_{N\geq 1}$ and $\vartheta^\MLK$ such that $\mu_N^{\MLK} \ed \mu_{\tH}$ and $\vartheta^{\MLK} \ed \E^\omega \nu_{\tD}$. Therefore, without loss, we will assume that $\mu_{\tH}$ and $\E^\omega \nu_{\tD}$ are defined on the same probability space and \eqref{eq:pt_conv_assumption} holds.

Since we have shown that each realization of the measure $\E^\omega \nu_{\tD}$ is determined by its moments, it follows that (see, for example, see \cite[Theorem 30.2]{Bil1}) almost surely the moments of $\mu_N^{\MLK}$ converge to those of $\vartheta^\MLK$ as $N \to \infty$, which implies that $\mu_{\tH} \Rightarrow \E^\omega \nu_{\tD}$ almost surely, as desired.
 
Let $m_N(r)$ be the $r$th moment of $\mu_N^{\MLK}$. Then 
\begin{align} \label{prop 3.1 eqn 2}
m_N(r) & = \frac{1}{2N}\mathrm{tr}\lp H^\MLK\rp^r =\frac{1}{2N}\sum_{\tp\in \mathcal{P}} P^\L(k_0,k_1)P^\L(k_1,k_2)\cdots P^\L(k_{2r-1},k_0) \cdot d_{k_1}^\MK  d_{k_3}^\MK\cdots  d_{k_{2r-1}}^\MK
\end{align}
where $\mathcal{P}$ is the set of paths 
$\tp = (k_0, k_1, \ldots, k_{2r-1})\in [2N]^{2r}$ such that
\begin{align*}
 |k_j - k_{j+1}| \leq L \mathrm{\;or\;} |k_j - k_{j+1}| \geq 2N-L \mathrm{\;for\;all\;} j \in [2r],
\end{align*}
where we adopt the convention that $k_{2r}=k_{0}$. Let 
\begin{align*}
\mathcal{P}' & = \{\tp\in \mathcal{P} : |k_j - k_{j+1}| \leq L \mathrm{\;for\;all\;} j \in [2r]\}.
\end{align*}
We claim that $\#(\mathcal{P}\setminus\mathcal{P}')$ is bounded above by a constant that depends only on $L$ and $r$. 

Observe that if $k_j \in [L, 2N- L]$, then 
no choice of $k_{j+1} \in [2N]$ satisfies 
$|k_j - k_{j+1}| \geq 2N-L$. Hence, as long as  the initial index $k_0 \in [2rL, 2N -2rL]$, any path in $\mathcal{P}$ can not satisfy the constraint $|k_j - k_{j+1}| \geq 2N-L $ for any $j$. Therefore, 
\[ \mathcal{P}\setminus\mathcal{P}' \subseteq \{\tp\in \mathcal{P} : 
k_0 \le 2rL \text{ or } k_0 \ge 2N - 2rL \}.\]
For each such bad choice of the initial index $k_0$, we have at most $(2L+1)^{2r-1}$ many paths in $\mathcal{P}$. Consequently,
\[ \#(\mathcal{P}\setminus\mathcal{P}') \le (4rL)(2L+1)^{2r-1}. \]

By the above estimate and the facts that $|d_k^\MK|\leq MK$ and $|P(k, l)| \le 1$, 
we obtain from \eqref{prop 3.1 eqn 2} that 
\begin{align*}
m_N(r) & =\frac{1}{2N}\sum_{\tp \in \mathcal{P}'} P^\L(k_0,k_1)P^\L(k_1,k_2)\cdots P^\L(k_{2r-1},k_0) \cdot d_{k_1}^\MK  d_{k_3}^\MK\cdots  d_{k_{2r-1}}^\MK + O\lp N^{-1}\rp.
\end{align*}
Lemma \ref{lem 2.2} and the boundedness of $d^\MK_k$ imply that \begin{align*}
m_N(r)  = \frac{1}{2N}\sum_{\tp \in \mathcal{P}'} \Pi^\L(k,k_1)\Pi^\L(k_1,k_2)\cdots \Pi^\L(k_{2r-1},k) \cdot d_{k_1}^\MK  d_{k_3}^\MK\cdots  d_{k_{2r-1}}^\MK + O\lp N^{-1} \rp.
\end{align*}
Let $\Theta_N$ be uniform on $[2N]$ as in \eqref{eq:pt_conv_assumption}, independent of all other randomness. Since $\Pi^\L(j,k) = 0$ for $|j-k| >L$, we can continue the computation to write \begin{align*}
& m_N(r)  = \frac{1}{2N}\sum_{k, k_1, \ldots, k_{2r-1} \in [2N]} \Pi^\L(k,k_1)\Pi^\L(k_1,k_2)\cdots \Pi^\L(k_{2r-1},k) \cdot d_{k_1}^\MK  d_{k_3}^\MK\cdots  d_{k_{2r-1}}^\MK + O\lp N^{-1}\rp\\
	& = \E_{\Theta_N} \sum_{k, k_1, \ldots, k_{2r-1} \in [2N]} \Pi^\L(\Theta_N,k_1)\Pi^\L(k_1,k_2)\cdots \Pi^\L(k_{2r-1},\Theta_N) \cdot d_{k_1}^\MK  d_{k_3}^\MK\cdots  d_{k_{2r-1}}^\MK + O\lp N^{-1}\rp,
	\end{align*}
where by $\E_{\Theta_N}$, we mean expectation over the randomness of $\Theta_N$ only. With the understanding that $ d_k^\MK = 0$ for $k <0$ or $k \geq 2N$, we utilize the Toeplitz structure of $\Pi^\L$ and change variables to obtain \begin{align}
	& m_N(r) = \E_{\Theta_N} \sum_{k_1, \ldots, k_{2r-1} \in \Z} \Pi^\L(0,k_1-\Theta_N)\Pi^\L(k_1-\Theta_N,k_2-\Theta_N)\cdots \Pi^\L(k_{2r-1}-\Theta_N,0) \notag\\
	& \hspace{300 pt }\cdot\;  d_{k_1}^\MK  d_{k_3}^\MK\cdots  d_{k_{2r-1}}^\MK+ O\lp N^{-1}\rp \notag\\
	& = \sum_{k_1, \ldots, k_{2r-1} \in \Z} \Pi^\L(0,k_1)\Pi^\L(k_1,k_2)\cdots \Pi^\L(k_{2r-1},0)  \cdot \E_{\Theta_N}  d_{k_1+\Theta_N}^\MK  d_{k_3+\Theta_N}^\MK\cdots  d_{k_{2r-1}+\Theta_N}^\MK + O\lp N^{-1} \rp. \label{prop 3.1 eqn 3}
\end{align}
We point out that the indices in the above sum can be restricted to  
$|k_j| \le 2Lr$ for each $j$, violation of which gives zero contribution to the sum.

The key step is now to show that as $N\to \infty$, \begin{align} \label{prop 3.1 eqn 4}
\E_{\Theta_N} \prod_{l=1}^r d_{k_{2l-1}+\Theta_N}^\MK \to \E^\omega \prod_{l=1}^r \varrho^\MK_{k_{2l-1}}.
\end{align}
Since 
\begin{align*}
 d_{k_l + \Theta_N}^\MK = 2\sum_{j=0}^{K-1}  b_{(j)}^\M \cos \Bigl(\frac{2\pi k\sigma_N(j)}{2N} + \frac{2\pi \Theta_N \sigma_N(j)}{2N}\Bigr).
\end{align*}
for each $k_1, k_2, \ldots, k_{2r-1} \in [-2Lr, 2Lr] $, there exists a continuous bounded function $f$ that is depends on $k_1, k_2, \ldots, k_{2r-1}$ but is independent of $N$ such that we can write \begin{align*} 
\prod_{l=1}^r d_{k_{2l-1}+\Theta_N}^\MK = f\left(\Bigl( b_{(j)}^\M, \frac{\sigma_N(j)}{2N}, \bb{\frac{\Theta_N \sigma_N(j)}{2N}}\Bigr)_{j \in[K]}\right).
\end{align*}
Under the assumption \eqref{eq:pt_conv_assumption} and by continuity of $f$, we have 
\begin{align*}
    f\left(\Bigl( b_{(j)}^\M, \frac{\sigma_N(j)}{2N}, \bb{\frac{\Theta_N \sigma_N(j)}{2N}}\Bigr)_{j \in[K]}\right) \to f\left(\Bigl(\varepsilon_j(\Gamma_j^\M)^{-1/\alpha}, \zeta_j, U_j\Bigr)_{j\in[K]}\right)
\end{align*}
Since $f$ is uniformly bounded, by the conditional dominated convergence theorem, we obtain that \begin{align}\label{prop 3.1 eqn 5}
\E_{\Theta_N} \prod_{l=1}^r d_{k_{2l-1}+\Theta_N}^\MK \to \E^\omega \prod_{l=1}^r  \breve{\varrho}^\MK_{k_{2l-1}}
\end{align}
where  \begin{align*}
\breve{\varrho}^\MK_k:= \sum_{j=0}^{K-1} \varepsilon_j(\Gamma_j^\M)^{-1/\alpha}\cos (2\pi (U_j + k\zeta_j))
\end{align*}
are the entries of the diagonal operator $\breve{\Lambda}^\MK = \mathrm{diag}((\breve{\varrho}_k^\MK)_{k \in \Z})$. For a fixed realization of $\omega $ and $(\varepsilon_j)_{j\geq 0}$ and for every $ k\in \Z$, we have \begin{align*}
\varepsilon_j(\Gamma_j^\M)^{-1/\alpha}\cos(U_j + k\zeta_j) \ed (\Gamma_j^\M)^{-1/\alpha}\cos(U_j + k\zeta_j).
\end{align*}
Hence, \begin{align*}
\E^\omega \prod_{l=1}^r \breve{\varrho}^\MK_{k_{2l-1}} = \E^\omega \prod_{l=1}^r \varrho^\MK_{k_{2l-1}},
\end{align*}
which, together with \eqref{prop 3.1 eqn 5}, yields \eqref{prop 3.1 eqn 4}. 
Therefore, we conclude that
\begin{align*}
m_N(r) & \to  \sum_{k_1, \ldots, k_{2r-1} \in \Z} \Pi^\L(0,k_1)\Pi^\L(k_1,k_2)\cdots \Pi^\L(k_{2r-1},0) \cdot \E^\omega \prod_{l=1}^r \varrho^\MK_{k_{2l-1}} = m(r),
\end{align*}
establishing Proposition~\ref{prop 3.1}.

\section{Approximation of the empirical spectral distribution by matrix truncations} \label{sec 4}

Our goal for this section is to prove Proposition \ref{prop 4.2}, i.e., to show that $$\lim_{L\to \infty} \limsup_N \E \dl(\mu_H, \mu_{\tH}) = 0$$ under the assumption that the truncation levels $M$, $L$ and $K$ are related as described in \eqref{eq:trunc_assump}. We shall also prove that the empirical spectral distributions of the matrices $D$ and $D^\circ$ are close.

\begin{lemma} \label{lem 4.2}
As $N\to\infty$, the following convergence holds in probability.
\begin{align*}
\dl(\mu_{T},2\mu_H-\delta_0) \to 0.
\end{align*}
\end{lemma}

We shall need the following version of the Hoffman-Weilandt inequality.
\begin{lemma}[{\cite[Corollary A.41]{Bai1}}] \label{lem 4.1}
Let $A$ and $B$ be two symmetric matrices of size $N$ and let $\|\cdot\|_\f$ denote the Fr\"obenius norm. Then,
\begin{align*} 
\dl^4(\mu_A,\mu_B) & \leq \dl^3(\mu_A,\mu_B) \leq \frac{1}{N} \|A-B\|_\f^2.
\end{align*}
\end{lemma}

We shall frequently be referring to the following result about slowly varying functions.

\begin{lemma}[{\cite[Proposition 2.3]{Res}}] \label{fact 1}
For any $t>0$,
\begin{align*}
\lim_{N \to \infty} \frac{N \ell (tc_N)}{c_N^\alpha} = 1.
\end{align*}
\end{lemma}

\begin{proof}[Proof of Lemma \ref{lem 4.2}]
From identity \eqref{eq:esd_toep_circ}, it suffices to show that \begin{align*}
\dl(\mu_{T},2\mu_H-\delta_0) \le 2 \dl(\mu_{PD^\circ P},\mu_{P DP}) \to 0
\end{align*}
in probability as $N\to\infty$. Using Lemma \ref{lem 4.1} and the fact that $\|P\|_\op = 1$, we obtain \begin{align*}
\dl^4(\mu_{PD^\circ P},\mu_{P DP}) & \leq \frac{1}{2N} \|P(D^\circ -D)P\|_\f^2 \leq \frac{1}{2N}\|D^\circ - D\|_\f^2 = \frac{1}{2N} \sum_{k=0}^{2N-1} (d_k^\circ - d_k)^2 = b_0^2.
\end{align*}
Fix $\varepsilon >0$. The above computation leads to \begin{align*}
\P (\dl(\mu_{PD^\circ P}, \mu_{PDP}) > \varepsilon) \leq \P (|b_0|> \varepsilon^2) = \frac{\ell(c_N\varepsilon^2)}{(c_N\varepsilon^2)^\alpha} = \frac{N\ell(c_N\varepsilon^2)}{c_N^\alpha}\frac{1}{N\varepsilon^{2\alpha}} \to 0
\end{align*}
where the last step follows from Lemma \ref{fact 1}.
\end{proof}

\subsection{Proof of Proposition \ref{prop 4.2}}

From the definition of $P$, one readily obtains the rate of decay of the entries: \begin{align} \label{eqn 2.7}
|P(k,l)| = \Big|2N\sin\Bigl(\frac{\pi(k-l)}{2N}\Bigr)\Big|^{-1} \leq \frac{C}{\min(|k-l|,2N-|k-l|)}; \qquad k \neq l \in [2N]
\end{align}
for some absolute constant $C$. This can be used to obtain an upper bound on the operator norm of $P^\L$.

\begin{lemma}\label{lem 2.3}
$\|P^\L\|_\op \leq C \log L$ for some constant $C>0$ independent of $N$ and $L$.
\end{lemma}
\begin{proof}
By \eqref{eqn 2.7}, we have that for any $k \in [2N]$,
\begin{align*}
\sum_{l=0}^{2N-1}|P^\L(k,l)| = \frac{1}{2} + 2 \sum_{l=1}^L |P^\L(k,k+l)| \leq C\sum_{l=1}^L\frac{1}{l} \leq C\log L.
\end{align*}
Since $P^\L$ is Hermitian, by H\"older's inequality we obtain that\begin{align*}
\|P^\L\|_{\op} \leq \sqrt{\|P^\L\|_1\|P^\L\|_\infty} = \|P^\L\|_\infty = \max_k \sum_{l=0}^{2N-1}|P^\L(k,l)| \leq C \log L.
\end{align*}
\end{proof}

We decompose the proof of Proposition \ref{prop 4.2} into three lemmas. In Lemmas \ref{lem 4.4}, \ref{lem 4.5} and \ref{lem 4.7}, we respectively show that the empirical spectral distribution of $H$ is close to that of $\fH$, which is close to that of $\sH$, which in turn is close to that of $\tH$. The proof of Proposition \ref{prop 4.2} then follows from these three lemmas and the triangle inequality.

\begin{lemma} \label{lem 4.4}
\begin{align*}
\lim_{L\to \infty} \limsup_N \E \dl(\mu_H, \mu_{\fH}) = 0.
\end{align*}
\end{lemma}

\begin{proof}
For two probability measures $\nu_1$ and $\nu_2$ on $\R$, define the Kolmogorov-Smirnov distance between them as \begin{align*}
\dks(\nu_1,\nu_2) = \sup_{t\in \R} |\nu_1((-\infty, t]) - \nu_2((-\infty, t])|.
\end{align*}
Note that $\dl(\nu_1,\nu_2) \leq \dks(\nu_1, \nu_2)$ (see, for example, \cite{Hub}). From the rank inequality (\cite[Theorem A.43]{Bai1}), we obtain that \begin{align*}
	\E\;\dks(\mu_H, \mu_{\fH}) & \leq  \frac{1}{2N}\E\;\mathrm{rank}(H-\fH) \leq \frac{1}{2N}\E\;\mathrm{rank}(D-D^\M)\\
	& \leq \frac{1}{2N}\E \Bigl[ 2N\sum_{j=0}^{N-1}\1_{\{|b_j| >M\}}\Bigr] \leq N\P\lp|a_0|\geq Mc_N\rp = \frac{N \ell (Mc_N)}{(Mc_N)^\alpha} \to M^{-\alpha},
\end{align*}
where in the last step we used Lemma \ref{fact 1}. Sending $M = L^{1/9}\to \infty$ completes the proof.
\end{proof}

\begin{lemma} \label{lem 4.5}
\begin{align*}
\lim_{L\to \infty} \limsup_N \E \dl ( \mu_{\fH}, \mu_{\sH}) = 0.
\end{align*}
\end{lemma}

\begin{proof}
We use Lemma \ref{lem 4.1} to write
\begin{align}
\dl^4(\mu_{\fH}, \mu_{\sH}) & \leq \frac{1}{2N}\|\fH-\sH\|_{\f}^2 \leq \frac{1}{2N} \lp \|(P - P^\L) D^\M P\|_{\f}^2 + \|P^\L D^\M (P - P^\L)\|_{\f}^2 \rp \notag \\
& \leq \frac{1}{2N} \lp \|(P - P^\L) D^\M\|_{\f}^2\|P\|_\op^2 + \|P^\L\|_\op^2\| D^\M (P - P^\L)\|_{\f}^2 \rp \label{lem 4.5 eqn 1}.
\end{align}
Further, $\|D^\M(P-P^\L)\|_\f = \|(P-P^\L)D^\M\|_\f.$ Using Lemma \ref{lem 2.3} we have from \eqref{lem 4.5 eqn 1} that for some universal constant $C>0$,
\begin{align}\label{lem 4.5 eqn 2}
\dl^4(\mu_{\fH}, \mu_{\sH}) & \leq \frac{C(\log L)^2}{N} \|(P-P^\L)D^\M\|_F^2.
\end{align}
From \eqref{eqn 2.7}, for any $k\in[2N]$ we have \begin{align*}
\|(P-P^\L) e_k\|_2^2 & = \sum_{j \in [2N]: L <|j-k|<2N-L} |P(j,k)|^2 \leq C\sum_{j=L}^{2N} \frac{1}{j^2} \leq \frac{C}{L}.
\end{align*} Hence,
\begin{align}\label{lem 4.5 eqn 3}
	 \|(P-P^\L)D^\M\|_\f^2 &  = \sum_{k=0}^{2N-1} \|(P-P^\L)D^\M e_k\|_2^2 = \sum_{k=0}^{2N-1} (d_k^\M)^2 \|(P-P^\L) e_k\|_2^2 \leq \frac{C}{L}\sum_{k=0}^{2N-1} (d_k^\M)^2.
\end{align}
By Parseval's theorem, we obtain  
\begin{align}\label{lem 4.5 eqn 4}
\sum_{k=0}^{2N-1} (d_k^\M)^2 & = 4N\sum_{j=0}^{N-1}(b_j^\M)^2. 
\end{align}
Indeed, recalling that $b_N = 0$, and defining $b_{2N} = b_0$ so that we have $b_j = b_{2N-j}$ for $j\in[N]$, we can write
\begin{align*}
	\sum_{k=0}^{2N-1} (d_k^\M)^2 &  = \sum_{k=0}^{2N-1} \Big(2 \sum_{j=0}^{N-1} b_j^\M \cos\Big(\frac{2\pi jk}{2N}\Big)\Big)^2 = \sum_{k=0}^{2N-1} \Big| \sum_{j=0}^{2N} b_j^\M \exp\Big(\frac{2\pi ijk}{2N}\Big)\Big|^2 \\
	& = \sum_{k=0}^{2N-1} \sum_{j,l = 0}^{2N} b_j^\M b_l^\M\exp\Big(\frac{2\pi i(j-l)k}{2N}\Big) = \sum_{j,l = 0}^{2N} b_j^\M b_l^\M \sum_{k=0}^{2N-1}  \exp\Big(\frac{2\pi i(j-l)k}{2N}\Big)\\
	& = 2N\sum_{j=0}^{2N} (b_j^\M)^2 = 4N \sum_{j=0}^{N-1}(b_j^\M)^2.
\end{align*}
Hence from \eqref{lem 4.5 eqn 2}-\eqref{lem 4.5 eqn 4} we obtain
\begin{align}
\limsup_N \E \dl^4(\mu_{\fH}, \mu_{\sH}) & \leq \limsup_N \frac{C( \log L)^2}{NL} \E \sum_{k=0}^{2N-1} (d_k^\M)^2 = \limsup_N \frac{C( \log L)^2}{L} \E \sum_{j=0}^{N-1} (b_j^\M)^2 \notag \\
 	& =\limsup_N \frac{CN(\log L)^2}{L} \E (b_0^\M)^2 \label{lem 4.5 eqn 5}
\end{align}
We now use the following result on the truncated moments of $a_0$ (\cite[Theorem VII.9.2]{Fel})
\begin{align*}
\lim_{t \to \infty} \frac{\E [|a_0|^2 \1_{\{|a_0| \leq t\}}]}{t^{2-\alpha} \ell(t)} = \frac{\alpha}{2-\alpha}.
\end{align*}
Together with Lemma \ref{fact 1}, we obtain that as $N\to \infty$, \begin{align}
N \E (b_0^\M)^2 = & N \E [b_0^2\1_{|b_0| \leq M}] + N M^2 \P( |b_0| \geq M) \notag\\
 & = \frac{\E [a_0^2 \1_{\{|a_0| \leq Mc_N\}}]}{c_N^2}\cdot \frac{c_N^\alpha}{M^{2-\alpha}\ell(Mc_N)} \cdot \frac{N \ell(Mc_N)}{c_N^\alpha} \cdot M^{2-\alpha} + NM^2 \frac{\ell(Mc_N)}{(Mc_N)^\alpha} \notag \\
 & \to \frac{2}{2-\alpha} M^{2-\alpha}. \label{eq:limit_expn}
\end{align}

Hence, from \eqref{lem 4.5 eqn 5}, \begin{align*}
\limsup_N \E \dl^4(\mu_{\fH}, \mu_{\sH}) \leq \frac{C (\log L)^2}{L}\frac{2}{2-\alpha}M^{2-\alpha} = \frac{C}{2-\alpha} (\log L)^2 L^{-(7+\alpha)/9}.
\end{align*}
The result now follows from Jensen's inequality and then taking $L\to \infty$.
\end{proof}

\begin{lemma} \label{lem 4.6}
There exists some constant $C$ such that for all $K > \frac{2}{\alpha}$, 
\begin{align*}
\limsup_N \sum_{j=K}^{N-1} \E( b_{(j)}^\M)^2 \leq C K^{1-2/\alpha}
\end{align*}
\end{lemma}

\begin{proof}
By an application of the continuous mapping theorem and the bounded convergence theorem, we obtain from  Lemma \ref{lem 3.3}  that 
$$\sum_{j=0}^{K-1} \E(b_{(j)}^\M)^2 \to \sum_{j=0}^{K-1} \E  (\Gamma_j^\M)^{-2/\alpha}.$$ This, together with \eqref{eq:limit_expn} yields that 
\begin{align}
    \limsup_N \sum_{j=K}^{N-1}\E (b_{(j)}^\M )^2 & \leq \limsup_N \sum_{j=0}^{N-1} \E (b_{(j)}^\M )^2 - \liminf_N \sum_{j=0}^{K-1} \E (b_{(j)}^\M )^2 \notag \\
    & = \lim_{N\to\infty} N \E (b_0^\M)^2 - \lim_{N\to\infty} \sum_{j=0}^{K-1} \E (b_{(j)}^\M )^2  = \frac{2}{2-\alpha}M^{2-\alpha} - \sum_{j=0}^{K-1} \E (\Gamma_j^\M)^{-2/\alpha}. \label{eq: bound_limsup}
\end{align}
Observing that $\Gamma_j \sim \mathrm{Gamma}(j+1,1)$ for each $j\geq 0$, one deduces that \begin{align*}
    \sum_{j=0}^\infty \E (\Gamma_j^\M)^{-2/\alpha} = \sum_{j=0}^\infty\int_0^\infty \min(t^{-2/\alpha}, M^2) \frac{t^j}{j!}e^{-t}dt = \int_0^\infty \min(t^{-2/\alpha}, M^2) dt = \frac{2}{2-\alpha}M^{2-\alpha}.
\end{align*}
Note that   $\E \Gamma_j^{-2/\alpha} < \infty$ if $j > 2/\alpha -1 $.
Continuing from \eqref{eq: bound_limsup},  we have, for $K > 2/\alpha,$
\begin{align*}
     \limsup_N \sum_{j=K}^{N-1}\E (b_{(j)}^\M )^2 & \leq \sum_{j=K}^\infty \E (\Gamma_j^\M)^{-2/\alpha} \leq \sum_{j=K}^\infty \E \Gamma_j^{-2/\alpha} = \sum_{j=K}^\infty \frac{\Gamma(j-\frac{2}{\alpha}+1)}{\Gamma(j+1)}.
\end{align*}
Let $k_0 = \lceil \frac{2}{\alpha}\rceil+1$. From the  recurrence relation $\Gamma(x+1) = x\Gamma(x)$ for any $x >0$, we estimate 
     \begin{align*}
    \sum_{j=K}^\infty \frac{\Gamma(j-\frac{2}{\alpha}+1)}{\Gamma(j+1)} & = \frac{\Gamma(k_0 - \tfrac{2}{\alpha})}{\Gamma(k_0)}\sum_{j=K}^\infty 
     \prod_{k=k_0}^j \frac{(k-\frac{2}{\alpha})}{ k}= \frac{\Gamma(k_0 - \tfrac{2}{\alpha})}{\Gamma(k_0)}\sum_{j=K}^\infty 
     \prod_{k=k_0}^j \Bigl(1-\frac{2}{\alpha k}\Bigr) \\
     & \leq \frac{\Gamma(k_0 - \tfrac{2}{\alpha})}{\Gamma(k_0)}\sum_{j=K}^\infty \exp\Bigl(-\frac{2}{\alpha}\sum_{k=k_0}^j \frac{1}{k}\Bigr)\leq C' \sum_{j=K}^\infty j^{-2/\alpha} \leq C K^{1-2/\alpha},
\end{align*}
for some constants $C'$ and $C$ that depend only on $\alpha$.
\end{proof}

\begin{lemma} \label{lem 4.7} 
We have
\begin{align*}
\lim_{L \to \infty} \limsup_N \E \dl( \mu_{\sH}, \mu_{\tH} )= 0
\end{align*}
\end{lemma}

\begin{proof}
From Lemma \ref{lem 2.3} we obtain \begin{align*}
& \frac{1}{2N}\|P^\L( D^\M -  D^\MK) P^\L\|_{\f}^2  \leq \frac{1}{2N}\|P^\L\|^4_{\op} \| D^\M -  D^\MK\|_{\f}^2 \leq \frac{C(\log L)^4}{N}\sum_{k=0}^{2N-1} \lp  d_k^\M -  d_k^\MK\rp^2.
\end{align*}
By Parseval's theorem, we have \begin{align}\label{lem 4.7 eqn 1}
\sum_{k=0}^{2N-1} \lp  d_k^\M -  d_k^\MK\rp^2 =4N \sum_{j=K}^{N-1} (b_{(j)}^\M)^2.
\end{align}
Indeed, the expressions on the left-hand side of \eqref{lem 4.7 eqn 1} and \eqref{lem 4.5 eqn 4} are similar, except that in \eqref{lem 4.7 eqn 1}, for each $k$, each of $b_{(0)}^\M, \ldots, b_{(K-1)}^\M$ has been replaced with 0. Thus, for some constant $C>0$ we have
\begin{align*}
	\frac{C(\log L)^4}{N}\sum_{k=0}^{2N-1} \lp  d_k^\M -  d_k^\MK\rp^2 
	& \leq C (\log L)^4 \sum_{j=K}^{N-1} (b_{(j)}^\M)^2
\end{align*}
and hence from Lemmas \ref{lem 4.1} and \ref{lem 4.6} we obtain that \begin{align*}
	\limsup_N \E \dl^4(\mu_{\sH}, \mu_{\tH} )
	& \leq \limsup_N \frac{1}{2N}\E \|P^\L(D^\M - D^\MK)P^\L\|_\f^2\\
    &\leq C (\log L)^4 \limsup_N \sum_{j=K}^{N-1}\E (b_{(j)}^\M)^2 \leq C(\log L)^4 K^{1-2/\alpha}.
\end{align*}
The proof is finished by using Jensen's inequality and taking $M = K = L^{1/9}\to \infty$.
\end{proof}

\section{Approximation of the spectral measure by operator truncations}\label{sec 5}
This section is divided in three parts. In the first part, we establish the properties of $\Delta^\omega$ as mentioned in Proposition \ref{prop 1.1}. Later, we prove Proposition \ref{prop 5.1} assuming relation \eqref{eq:trunc_assump}, i.e., $M = K = L^{1/9}$, we establish that 
\begin{align*}
\lim_{L\to \infty} \E \dl(\E^\omega\nu_{\tD}, \E^\omega\nu_{\Delta}) = 0.
\end{align*}
In the final section, we establish a connection between the spectral measures of $\Delta$ at $e_0$ and the unit vector $u = \sqrt{2}\Pi e_0$ by relating the Stieltjes transform of the two measures.

\subsection{Proof of Proposition \ref{prop 1.1}}
To lighten our notation, we will retain the superscript of $\omega$ only with $\P$ and $\E$ and drop it from $\varrho^\omega, \Lambda^\omega, \Delta^\omega, \D^\omega$, etc. if there is no scope of confusion.   Throughout the proof, we fix a realization of $\omega \in \Omega_1$.

\subsubsection{Proof of part (a)}
    Since $\omega \in \Omega_1$, we have that $\sum_{j=0}^{\infty}\Gamma_j^{-2/\alpha} < \infty$. By Kolmogorov's two-series theorem, $\P^\omega$-almost surely, $\varrho_k$ is finite for each $k \in \Z$.

    To prove that $\P^\omega(\sum_{k \in \Z} \varrho_k^2/(1+k^2) < \infty) =1$, we define, for $n, k \in \N$, \begin{align*}
    \xi_{k,n} : = \sum_{j=0}^n \Gamma_j^{-1/\alpha} \cos (2\pi(U_j + k\zeta_j)).
    \end{align*}
    Note that $(\xi_{k,n})_{n\geq 1}$ is a martingale adapted to the filtration $(\mathcal{G}_n)_{n\geq 0}$ where $\mathcal{G}_n$ is the $\sigma$-algebra generated by $U_0, \ldots, U_n$. Further, since $(U_j)_{j\geq 0}$ are independent we have \begin{align*}
	\E^\omega \xi_{k,n}^2 & =  \E^\omega \sum_{j=0}^n \Gamma_j^{-2/\alpha}\cos^2(2\pi(U_j+k\zeta_j)) \leq \sum_{j=0}^n \Gamma_j^{-2/\alpha}.
    \end{align*}
    Hence, \begin{align*}
    \sup_{n\geq 0} \E^\omega \xi_{k,n}^2 \leq  \sum_{j=0}^\infty \Gamma_j^{-2/\alpha} < \infty,
    \end{align*}
    i.e., $(\xi_{k, n})_{n \ge 0}$ is a $L^2$-bounded martingale and 
 thus by the martingale convergence theorem, 
    \begin{align*}
    \E^\omega \varrho_k^2 = \lim_{n \to \infty} \E^\omega \xi_{k,n}^2 \leq  \sum_{j=0}^\infty \Gamma_j^{-2/\alpha} < \infty. 
    \end{align*}
    Further, the upper bound in the above display is uniform over $k \in \Z$. Hence, \begin{align*}
        \E^\omega \sum_{k\in \Z} \frac{1}{k^2+1}  \varrho_k^2 < \infty
    \end{align*}
    from which the $\P^\omega$-almost sure finiteness of $\sum_{k \in \Z} \varrho_k^2/(1+k^2)$ follows.

\subsubsection{Proof of part (b)}
We already explained in the introduction that $\P^\omega$-almost surely, $\| \Pi \Lambda e_k \|_2 < \infty$ for all $k$. We fix a realization of $(U_j)_{j \ge 0}$ for which $\| \Pi \Lambda e_k \|_2 < \infty$ for all $k$. 

    One easily checks that $\Delta$ is Hermitian. To show that $\Delta$ is self-adjoint on the domain $\D$, we proceed along the lines of \cite[Proposition VIII.1]{Ree}. Let $\chi_n \in \lz$ be defined by \begin{align}\label{eq: def_chi}
    (\chi_n)_j&  = \begin{cases}
    1 \qquad \mathrm{if\;} |j|\leq n,\\
    0 \qquad \mathrm{otherwise}.
    \end{cases}
    \end{align}
    We need to show that the domain of the adjoint $\Delta^*$, denoted by $\D^*$, is equal to $\D$. Since $\D\subseteq \D^*$, it suffices to show the other containment. To that extent, let $v\in \D^*$. Then, by definition, $\|\Delta^* v\|_2 < \infty$. By the monotone convergence theorem, 
    \begin{align*}
    \|\Delta^* v\|_2 & = \lim_{n\to\infty}\|\chi_n\circ(\Delta^* v)\|_2
    \end{align*}
    where $\circ$ denotes entrywise product. Continuing, we have  \begin{align*}
    \|\Delta^* v\|_2 & = \lim_{n\to \infty} \sup_{\|w\|_2 = 1} |\langle w,\chi_n\circ(\Delta^* v)\rangle| = \lim_{n\to \infty} \sup_{\|w\|_2 = 1} |\langle \chi_n\circ w,\Delta^* v\rangle|
    \end{align*}
    Since $\chi_n\circ w$ has finitely many non-zero entries and $\|\Lambda\Pi e_k\|_2 < \infty$ for all $k \in \Z$, we have that $\|\Lambda \Pi(\chi_n\circ w)\|_2 <\infty$ and hence $\chi_n\circ w \in \D$. Thus, the above display can be written as \begin{align*}
	\|\Delta^* v\|_2 & = \lim_{n\to \infty}\sup_{\|w\|_2 = 1} |\langle \Delta (\chi_n\circ w),v\rangle| = \lim_{n\to \infty}\sup_{\|w\|_2 = 1} |\langle \Lambda\Pi (\chi_n\circ w),\Pi v\rangle|,
	\end{align*}
    where in the last step we used the fact that $\Pi$ is self-adjoint. Continuing,
	\begin{align*}
	\|\Delta^* v\|_2 & = \lim_{n\to \infty}\sup_{\|w\|_2 = 1} \Big|\sum_{k\in \Z} \varrho_k(\Pi(\chi_n\circ w))_k\overline{(\Pi v)_k}\Big| = \lim_{n\to \infty}\sup_{\|w\|_2 = 1} |\langle \Pi (\chi_n\circ w),\Lambda \Pi v\rangle|\\
	& = \lim_{n\to \infty}\sup_{\|w\|_2 = 1} |\langle \chi_n\circ w,\Delta v\rangle| = \lim_{n\to \infty}\sup_{\|w\|_2 = 1} |\langle w,\chi_n \circ (\Delta v)\rangle| \\
	& = \lim_{n\to \infty} \|\chi_n \circ (\Delta v)\|_2 = \|\Delta v\|_2 < \infty.
    \end{align*}
    Thus, $v \in \D$, concluding the proof.

\subsection{Proof of Proposition \ref{prop 5.1}}

As in the previous section, we fix a realization of $\omega \in \Omega_1$. Recall that  $\cC$ is the set of vectors in $\lz$ with finite support. The proof is decomposed into a few steps. In Lemma \ref{lem core} we show that $\cC$ is a core for $\Delta$. In Lemma \ref{lem strong_conv}, we show that  $$\tD v \to \Delta v, \quad \text{ for all } v \in \cC, $$ 
relying on a technical estimate from Lemma~\ref{lem key_estimate}.  We then invoke \cite[Theorem VIII.25(a)]{Ree}  to establish the convergence of $\tD$ to $\Delta$ in the strong resolvent sense, i.e., for all $v \in \lz$ and $z \in \C\setminus \R$  \begin{align*}
(\tD -zI)^{-1}v \to (\Delta - zI)^{-1}v.
\end{align*} 
Taking $v = e_0$ in the above display and then taking inner products with $e_0$, we obtain that the Stieltjes transform of $\nu_{\tD}$ converges to that of $\nu_\Delta$ and hence $\nu_{\tD} \Rightarrow \nu_\Delta$ as $L\to \infty$. Utilizing the independence of $U$ from $\omega$, we have $\E^\omega \nu_{\tD} \Rightarrow \E^\omega \nu_\Delta$, or equivalently, $\dl(\E^\omega\nu_{\tD},\E^\omega\nu_\Delta) \to 0$.  Since the L\'evy distance between two probabilities is bounded above by 1, the result follows from the dominated convergence theorem.

\begin{lemma}\label{lem core}
$\cC$ is a core for $\Delta$ and $\tD$ for any $M,K$, and  $L$.
\end{lemma}

\begin{proof}
Since $\tD$ is a bounded operator and $\cC$ is dense in $\lz$, $\cC$ is a core of $\tD$. Next, we consider the case of $\Delta$.
We already showed that $e_k \in \D$ for every $k \in \Z$, which implies that $\cC \subseteq \D$. To prove the lemma, we thus need to show that the closure of the graph of $\Delta$ on $\cC$ contains the graph of $\Delta$ on $\D$.

Therefore, we need to show that for each $v\in \D$, there exists a sequence $(v_n)_{n\geq 1}$ in $\cC$  such that $(v_n, \Delta v_n) \to (v, \Delta v)$ on $\lz \times \lz$.

 Fix $v\in \D$ and consider the sequence $v_n = \chi_n\circ v$, where $\chi_n$ is as in \eqref{eq: def_chi}. Clearly, $v_n \to v$ in $\lz$. We claim that  $\Delta v_n \to \Delta v$. Indeed, for any $k\in \Z$, since $\Delta$ is self-adjoint, \begin{align*}
  \langle e_k, \Delta v_n \rangle & = \langle \Delta e_k, v_n \rangle \to \langle \Delta e_k, v \rangle = \langle e_k, \Delta v \rangle,
 \end{align*}
 where the convergence follows from Cauchy-Schwarz: since $\|\Delta e_k\|_2 < \infty$, \begin{align*}
 |\langle \Delta e_k, v - v_n\rangle |  \leq \|\Delta e_k\|_2\|v - v_n\|_2 \to 0.
 \end{align*}
 \end{proof}

\begin{lemma}\label{lem key_estimate}
Fix $\omega \in \Omega_1$. Let $$\rho_{k,K} : = \sum_{j=K}^\infty (\Gamma_j^\M)^{-1/\alpha} \cos (2\pi(U_j + k\zeta_j)).$$
Then $\P^\omega$-almost surely,
\begin{align*}
\sum_{k\in \Z} \frac{1}{1+k^2} \sup_{K \geq 0} \rho_{k,K}^2 < \infty.
\end{align*}
\end{lemma}

\begin{proof}
We write $$\rho_{k,K}  = \sum_{j=K}^\infty (\Gamma_j^\M)^{-1/\alpha} \eta_{k,j} \theta_{k,j},$$ where $\eta_{k,j}: = \sgn(\cos(2\pi (U_j + k\zeta_j)))$ and $\theta_{k,j}: = |\cos(2\pi(U_j + k\zeta_j))|$. Note that for a fixed $k$,  $\eta_k =(\eta_{k,j})_{ j\geq 0}$ are i.i.d.\ Rademacher variables, independent of $\theta_k = (\theta_{k,j})_{j\geq 0}$.



Denoting by $\E_{\eta_k}$, the expectation with respect to $\eta_k$ only, it suffices to show that
\begin{align} \label{lem 5.5 eqn 2}
\sup_{k\in \Z}\E_{\eta_k} \sup_{K\geq 0} \rho_{k,K}^2 < C(\omega),
\end{align}
 for some finite constant $C(\omega)$, independent of $U, k,$ and $K$,
since this would imply that \begin{align*}
\E^\omega \sum_{k\in \Z} \frac{1}{1+k^2} \sup_{K \geq 0}\rho_{k,K}^2 < \infty
\end{align*}
from which the statement of the lemma follows.

We now proceed to prove \eqref{lem 5.5 eqn 2}. The proof will be split in two steps; in the first step, we show that
\begin{align}\label{lem 5.5 eqn 3}
\sup_{k\in \Z}\E_{\eta_k} \sup_{K\geq 0}|\rho_{k,K}| < C_1(\omega).
\end{align}
The second step involves showing the concentration of $\sup_{K\geq 0}|\rho_{k,K}|$ around its mean:
\begin{align}\label{lem 5.5 eqn 4}
\sup_{k \in \Z} \E_{\eta_k} \Bigl(\sup_{K \geq 0} |\rho_{k,K}| - \E_{\eta_k} \sup_{K\geq 0}|\rho_{k,K}| \Bigr)^2 < C_2(\omega).
\end{align}
\eqref{lem 5.5 eqn 2} follows from \eqref{lem 5.5 eqn 3} and \eqref{lem 5.5 eqn 4}. We fix $k \in \Z$ and prove the bounds \eqref{lem 5.5 eqn 3} and \eqref{lem 5.5 eqn 4} uniformly in $k$. For notational simplicity, we will write $\E_{\eta}$ instead of $\E_{\eta_k}$.

\begin{itemize}
\item \textit{Proof of \eqref{lem 5.5 eqn 3}}: It is useful to view $\sup_{K \geq 0} |\rho_{k,K}|$ as the supremum of a Rademacher process, where we treat $(\Gamma_j^\M)_{j \ge 0}$ and $(\theta_{k, j})_{j \ge 0}$ as deterministic sequences. Indeed, 
let $\A  =  \A^+ \cup \A^-$, where \begin{align*}
\A^+ = \A_k^+ := \bigcup_{K\geq 0}\Big\{ \sum_{j=K}^\infty (\Gamma_j^\M)^{-1/\alpha} \theta_{k,j}e_j\Big\} \subset \lz, \quad \text{ and } \A^-: = -\A^+.
\end{align*} 
  Then $(\langle v, \eta_k \rangle)_{v\in \A}$ is a Rademacher process and $\sup_{K \geq 0} |\rho_{k,K}| = \sup_{v \in \A} \langle v, \eta_k \rangle.$ Note that   $(\langle v, \eta_k \rangle)_{v\in \A}$ is a  mean-zero  subgaussian process with respect to the $\ell^2$-norm on $\A$. Therefore, by Dudley's integral inequality (see, for example, \cite[Theorem 8.1.3]{Ver}), there exists some absolute constant $C>0$ such that 
\begin{align*}
	\E_\eta \sup_{K \geq 0} |\rho_{k,K}| & = \E_\eta \sup_{v\in \A} \langle v, \eta_k \rangle \leq C\int_0^{\text{diam}{(\A)}} \sqrt{\log \mathcal{N}(\A, \|\cdot\|_2, \varepsilon)} d\varepsilon,
\end{align*}
where $\mathcal{N}(\A, \|\cdot\|_2, \varepsilon)$ is the covering number of $\A$, i.e.,  is the smallest number of  $\ell^2$-balls of radius $\varepsilon$ required to cover $\A$. Since $|\theta_{k,j}|\leq 1$, we have that $$\text{diam}(\A) = \sup_{v\in \A}\|v\|_2 \leq \Bigl(\sum_{j=0}^\infty (\Gamma_j^\M)^{-2/\alpha}\Bigr)^{1/2} \leq  \Bigl(\sum_{j=0}^\infty \Gamma_j^{-2/\alpha}\Bigr)^{1/2} =: C'(\omega)< \infty.$$ 
Also, as $\Gamma_j/j \to 1$, for all $K$ sufficiently large,
\begin{align*}
\sum_{j=K}^\infty \Gamma_j^{-2/\alpha} \leq 2 \sum_{j=K}^\infty j^{-2/\alpha} \leq 3 K^{1 - 2/\alpha}.
\end{align*}
Take $K_0$ sufficiently large such that $3K_0^{1 - 2/\alpha} \le \varepsilon$, i.e., $K_0 > (\varepsilon/3)^{-\alpha/(2 - \alpha)} $. Then 
the subset $\A \cap \mathrm{span}\{e_j: j \ge K_0 \}$ can be covered by the $\ell^2$-ball of radius $\varepsilon$ centered at the origin, and each of the remaining  $2K_0$ points of $\A $ can be covered trivially by one  $\ell^2$-ball of radius $\varepsilon$. Hence, we obtain 
\[ \mathcal{N}(\A, \|\cdot\|_2, \varepsilon) \leq C\varepsilon^{-\alpha/(2-\alpha)}, \]
for some constant $C$ and thus, 
\begin{align*}
\E_\eta \sup_{K \geq 0} |\rho_{k,K}| \leq C \int_0^{C'(\omega)} \sqrt{|\log \varepsilon|} d\varepsilon,
\end{align*}
which is a finite quantity that does not depend on $k$ and hence, \eqref{lem 5.5 eqn 3} follows.

\item \textit{Proof of \eqref{lem 5.5 eqn 4}}:
     We note that $\sup_{v\in \A} \langle v, \cdot \rangle$ is a convex Lipschitz function, with Lipschitz constant $\text{diam}(\A)$. Thus,  by Talagrand's concentration inequality (\cite[Theorem 5.2.16]{Ver}),  for any $t>0$  
     \begin{align*}
\P_\eta \Bigl(\Big| \sup_{v\in \A} \langle v, \eta_k \rangle - \E_\eta \sup_{v\in \A} \langle v, \eta_k \rangle \Big| > t\Bigr) & \leq C \exp\Bigl(-\frac{t^2}{C \text{diam}(\A)^2}\Bigr),
\end{align*}
for some universal constant $C>0$, and hence \begin{align*}
\E_\eta \Bigl( \sup_{v\in \A} \langle v, \eta_k \rangle - \E_\eta \sup_{v\in \A} \langle v, \eta_k \rangle \Bigr)^2 \leq C \text{diam}(\A) \leq C'(\omega).
\end{align*}
Since the above bound is uniform over $k$, \eqref{lem 5.5 eqn 4} follows by taking supremum over $k$.
\end{itemize}
\end{proof}

To establish the convergence of $\tD$ to $\Delta$ on $\cC$, we shall need the following estimate on $\|\Pi^\L\|_\op$.

 \begin{lemma}\label{lem pi_L bounded}
$\|\Pi^{\L}\|_\op \leq C \log L$ for some constant $C>0$ independent of $L$.
\end{lemma}
\begin{proof}
The proof follows by using arguments analogous to Lemma \ref{lem 2.3}. 
\end{proof}

\begin{lemma}\label{lem strong_conv}
$\tD v\to \Delta v$ for all $v \in \cC$.
\end{lemma}
\begin{proof}
It suffices to show that for all $k \in \Z$, $\tD e_k\to \Delta e_k$. For simplicity, we demonstrate the proof for $k=0$. The argument for other values of $k$ is similar. We have 
\begin{align} 
	\|\Delta e_0  - \tD e_0   \|^2 & \leq 2\|\Pi\Lambda\Pi e_0 -  \Pi\Lambda^\MK\Pi e_0\|^2 + 2\|\Pi\Lambda^\MK\Pi e_0  - \Pi^\L\Lambda^\MK\Pi^\L e_0\|^2 \label{lem 5.6 eqn 0}\\
	& \leq  2\|(\Lambda-\Lambda^\MK)\Pi e_0\|^2 + 2\|\Pi\Lambda^\MK\Pi e_0  - \Pi^\L\Lambda^\MK\Pi^\L e_0\|^2 \notag\\
	& \leq 4\|(\Lambda - \Lambda^\M)\Pi e_0\|^2 + 4\|(\Lambda^\M - \Lambda^\MK)\Pi e_0\|^2 \label{lem 5.6 eqn 1}\\
	& \qquad + 4\|(\Pi - \Pi^\L)\Lambda^\MK\Pi e_0\|^2 + 4\|\Pi^\L\Lambda^\MK(\Pi-\Pi^\L)e_0\|^2, \label{lem 5.6 eqn 2}
\end{align}
where we used the fact that $\|\Pi\|_\op = 1$ in the second step. We claim that each of the terms in \eqref{lem 5.6 eqn 1} and \eqref{lem 5.6 eqn 2} vanishes as we take $L \to \infty$.

\begin{itemize}
\item \textit{First term in \eqref{lem 5.6 eqn 1}}: We have 
\begin{align*}
\|(\Lambda - \Lambda^\M)\Pi e_0\|^2 & = \sum_{k \in \Z} |(\Lambda - \Lambda^\M)(k,k) \Pi(k,0)|^2 \\
	& \leq C\sum_{k\in \Z}\frac{1}{1 + k^2}\Bigl[\sum_{j=0}^\infty(\Gamma_j^{-1/\alpha}-M) \1_{\{\Gamma_j \leq M^{-\alpha}\}} \cos(2\pi(U_j + k \zeta_j))\Bigr]^2\\
        & \leq C\sum_{k\in \Z}\frac{1}{1 + k^2} \Bigl[\sum_{j=0}^\infty \Gamma_j^{-1/\alpha} \1_{\{\Gamma_j \leq M^{-\alpha}\}}\Bigr]^2 \\
        & \leq C\sum_{k\in \Z}\frac{1}{1 + k^2} \cdot \sum_{j=0}^\infty \Gamma_j^{-2/\alpha} \cdot \sum_{j=0}^\infty \1_{\{\Gamma_j \leq M^{-\alpha}\}},
\end{align*}

where the final step follows from Cauchy-Schwarz. The claim follows since both $\sum_{k\in \Z}\frac{1}{1 + k^2}$ and  $\sum_{j=0}^\infty \Gamma_j^{-2/\alpha}$ are finite, and that $ \sum_{j=0}^\infty \1_{\{\Gamma_j \leq M^{-\alpha}\}}  = 0$  for all $M > \Gamma_0^{-1/\alpha}$. 

\item \textit{Second term in \eqref{lem 5.6 eqn 1}}: We write 
\begin{align} \label{lem 5.6 eqn 3}
	\|(\Lambda^\M - \Lambda^\MK)\Pi e_0\|^2 & = \sum_{k\in \Z} \Big|\sum_{j=K}^\infty (\Gamma_j^\M)^{-1/\alpha} \cos (2\pi(U_j + k\zeta_j)) \Pi(k,0)\Big|^2 \leq C\sum_{k\in \Z}\frac{1}{1+k^2} \rho_{k,K}^2,
\end{align} 
where we borrowed the notation $\rho_{k,K}$ from Lemma \ref{lem key_estimate}. Note that for each $k$, almost surely, $\rho_{k,K} \to 0$ as $K \to \infty$. Indeed, for $K=M$ sufficiently large,  $\Gamma_j^\M = \Gamma_j$, so we have \begin{align*}
\rho_{k,K} = \sum_{j=K}^\infty \Gamma_j^{-1/\alpha} \cos(2\pi(U_j + k\zeta_j)),
\end{align*}
which converges to $0$ as $K \to \infty$ as it is the tail of a  convergent series by Proposition \ref{prop 1.1}(a). Now applying the dominated convergence theorem on \eqref{lem 5.6 eqn 3} with respect to the measure $\sum_{k \in \Z}(1+k^2)^{-1}\delta_k$ on $\Z$, and noting that dominating condition is satisfied owing to Lemma~\ref{lem key_estimate}, we conclude \eqref{lem 5.6 eqn 3} vanishes as $K \to \infty$.

\item \textit{First term in \eqref{lem 5.6 eqn 2}}: We have that $\max_{k\in \Z} |\varrho_k^\MK| \leq 2MK$. For brevity, we let $v = \Lambda^\MK\Pi e_0$. Thus, for some universal constant $C>0$ and any $k\in\Z$,\begin{align*}
	|v_k| & = |\Lambda^\MK(k,k)\Pi(k,0)| \leq \frac{CMK}{1+|k|}.
\end{align*}
Let $w = \chi_{\lfloor\sqrt{L}\rfloor} \circ v$, where $\chi_n$ is as defined in \eqref{eq: def_chi}. Then 
\begin{align*}
	\|v-w\|_2^2 & = \sum_{|k|> \lfloor\sqrt{L}\rfloor}|v_k|^2 \leq \frac{CM^2K^2}{\sqrt{L}}.
\end{align*}
By Cauchy-Schwarz, we have 
\begin{align*}
	\|w\|_1^2 & \leq \sqrt{L}\|v\|_2^2 \leq C\sqrt{L}M^2K^2.
\end{align*}
From Lemma \ref{lem 2.3} we obtain $\|\Pi - \Pi^\L\|_\op \leq C \log L$. Owing to the Toeplitz structure of $\Pi$ and $\Pi^\L$, we have that for any $j \in \Z$, \begin{align} \label{lem 5.6 eqn 5}
\|(\Pi - \Pi^\L) e_j\|_2^2 = \|(\Pi - \Pi^\L)e_0\|_2^2 = \sum_{|k| >L} |\Pi(k,0)|^2 \leq \frac{C}{L}.
\end{align}
Hence, \begin{align*}
	\|(\Pi - \Pi^\L)v\|_2^2 & \leq 2\|(\Pi-\Pi^\L)w\|^2 + 2\|(\Pi-\Pi^\L) (v-w)\|^2\\
	& \leq 2\Big\|\sum_{|j|\leq \sqrt{L}} w_j (\Pi - \Pi^\L) e_j\Big\|_2^2 + 2\|\Pi-\Pi^\L\|_\op^2 \|v-w\|_2^2\\
	& \leq 2\Bigl(\sum_{|j|\leq \sqrt{L}}|w_j|\|(\Pi - \Pi^\L)e_j\|_2\Bigr)^2 + C\frac{(\log L)^2 M^2K^2}{\sqrt{L}}\\
	& \leq \frac{C}{L}\|w\|_1^2 + C\frac{(\log L)^2 M^2K^2}{\sqrt{L}} \leq C\frac{M^2K^2 (\log L)^2}{\sqrt{L}}= C (\log L)^2L^{-1/18} \to 0
\end{align*}
as $L \to \infty$.
\item \textit{Second term in \eqref{lem 5.6 eqn 2}}: From Lemma \ref{lem pi_L bounded} and \eqref{lem 5.6 eqn 5}, we have \begin{align*}
\|\Pi^\L\Lambda^\MK(\Pi-\Pi^\L)e_0\|_2^2 & \leq \|\Pi^\L\|_\op^2 \|\Lambda^\MK\|_\op^2 \|(\Pi-\Pi^\L)e_0\|_2^2 \leq C \frac{(\log L)^2 M^2K^2}{L} = C(\log L)^2 L^{-5/9}
\end{align*}
which approaches 0 as $L \to \infty$.
\end{itemize}
Since all terms in \eqref{lem 5.6 eqn 1} and \eqref{lem 5.6 eqn 2} vanish in the limit $L\to \infty$, the proof is complete.
\end{proof}

\subsection{Connecting the spectral measures of \texorpdfstring{$\Delta$}{Delta} at \texorpdfstring{$e_0$}{e0} and at \texorpdfstring{$\sqrt{2}\Pi e_0 $}{sqrt(2) Pi e0}}

One easily observes that $\|\Pi e_0\|_2 = 1/2$, so $u = \sqrt{2}\Pi e_0$ is a vector of unit norm. In the following lemma, we relate the probability measures $\nu_\Delta = \nu_{\Delta, e_0}$ and $\nu_{\Delta,u}$. This paves the way to provide a simple description of the weak limit of the empirical spectral distribution of $T_N$ in Theorem \ref{main_thm}.

\begin{lemma}\label{lem spec_measure_at_u}
    $\P^\omega$-almost surely, $2\nu_\Delta - \delta_0 = \nu_{\Delta, u}$.
\end{lemma}

\begin{proof}
We fix a realization such that $\Delta$  is self-adjoint.
   Note that it is equivalent to show that the corresponding Stieltjes transforms agree, that is, for all $z \in \C\setminus\R$ \begin{align} \label{stielt}
    2\big  \langle e_0,  ( \Delta  - zI)^{-1} e_0  \big \rangle + \frac{1}{z} &= \big \langle u,  ( \Delta  - zI)^{-1} u   \big \rangle.
\end{align}
We prove \eqref{stielt} by showing it first for the truncated operator $\Pi\Lambda^\MK\Pi$, which we denote by $\dD$ for simplicity. For $z \not \in \R$ with $|z| > MK \ge \|\Lambda^\MK\|_\op \geq \|\dD\|_\op$, (note that $\|\Pi\|_\op = 1$), the following power series is convergent
\begin{align*}
    (\dD - zI)^{-1} = -\sum_{k=0}^\infty \frac{(\dD)^k}{z^{k+1}},
\end{align*}
where we use the convention that $(\dD)^0 = I$. Noting that $\Pi(I-\Pi) = (I- \Pi)\Pi=0$,  this yields 
 \begin{align*}
\big \langle \Pi e_0,  ( \dD  - zI)^{-1} (I -\Pi)e_0  \big \rangle = \big \langle (I- \Pi) e_0,  ( \dD  - zI)^{-1} \Pi e_0  \big \rangle = 0,
  \end{align*}
  and
   \begin{align*}
\big \langle (I- \Pi) e_0,  ( \dD - zI)^{-1} (I -\Pi)e_0  \big \rangle = -  \frac{ \| (I - \Pi) e_0 \|_2^2}{z} =  -  \frac{1}{2 z}. 
  \end{align*}
Thus,
   \begin{align}
2\big  \langle e_0,  ( \dD - zI)^{-1} e_0  \big \rangle + \frac{1}{z} &=  2 \big \langle \Pi e_0,  ( \dD - zI)^{-1} \Pi e_0  \big \rangle  - 2\cdot\frac{1}{2z} + \frac{1}{z} =  \big \langle u,  ( \dD - zI)^{-1} u   \big \rangle. \label{stielt_bounded}
  \end{align}
By analytic continuation, \eqref{stielt_bounded} holds true for all $z \in \C\setminus \R$. Thus, \eqref{stielt} follows from \eqref{stielt_bounded} if we can show that $\dD \to  \Delta$ in the strong resolvent sense as $M=K \to \infty$. It suffices to show that  $\dD v \to \Delta v$ for each $v \in \mathcal{C}$. Indeed, in Lemma \ref{lem strong_conv} we proved this for $ v = e_0$ by showing that the first term in \eqref{lem 5.6 eqn 0} vanishes, and remarked that a similar computation holds for $v = e_k$, $k\neq 0$, thereby extending the result to all vectors in $\mathcal{C}$. Consequently, thanks to \cite[Theorem VIII.25(a)]{Ree}, the convergence of $\dD$ to $\Delta$ holds in the strong resolvent sense. 
\end{proof}

\section{Properties of Limiting Spectral Distribution}\label{sec 6}
In this section, we will prove Theorem \ref{thm 1.3}.
Throughout all parts, we fix a realization of $\omega \in \Omega_0$.

\subsection{Proof of parts (a) and (c)}

We have 
\begin{align*}
\Bigl(\Gamma_j^{-1/\alpha}\cos(2\pi(U_j + k\zeta_j))\Bigr)_{j,k\in \Z} \ed \Bigl(-\Gamma_j^{-1/\alpha}\cos(2\pi(U_j + k\zeta_j))\Bigr)_{j,k\in \Z}
\end{align*}
and hence $\varrho_k \ed -\varrho_k$ for each $k \in \Z$. Thus, the probability measures $\E^\omega \nu_\Delta$ and $\E^\omega \nu_{-\Delta}$ are the same, and hence these distributions are symmetric around 0. It follows then that $\nu_\ct = 2\E^\omega \nu_\Delta -\delta_0$ also exhibits the same property, thereby concluding the proof of part (a).

In part (c), since we assume that $0<\alpha<1$, we have\begin{align*}
   \|\Lambda\|_\op = \sup_{k\in \Z}|\varrho_k| \leq 2\sum_{j=0}^\infty \Gamma_j^{-1/\alpha}< \infty.
\end{align*}
Since $\|\Delta\|_\op \leq \|\Lambda\|_\op$, it follows that $\Delta$ is a bounded operator. Hence,   $\sigma(\Delta)$, the spectrum of $\Delta$, is  contained in the interval $[-2\sum_{j=0}^\infty \Gamma_j^{-1/\alpha}, 2\sum_{j=0}^\infty \Gamma_j^{-1/\alpha}]$ that does not depend on $U$. Part (c) then follows as 
$\mathrm{supp}(\nu^\omega_\ct) \subseteq \sigma(\E^\omega \Delta)$.

\subsection{Proof of part (b)}

We begin by quoting a result regarding the limiting distribution of heavy-tailed circulant matrices. Recall the $2N \times 2N$ symmetric circulant matrix $G$ from \eqref{eq:matrix G}.

\begin{theorem}[{\cite[Theorem 3.3(b)]{Bos1} }]\label{circular_lim_dist} For $0 < \alpha <2$, 
    $\mu_G \Rrightarrow \nu_\cc$ where $\nu^\omega_\cc$ is the law of the random variable  $2\sum_{j=0}^\infty \Gamma_j^{-1/\alpha} \cos(2\pi U_j)$ conditioned on $\omega$. 
\end{theorem}

In Section \ref{sec 2.1} we showed that the matrix $T$ is the $N\times N$ principal submatrix of the circulant matrix $G$. By Cauchy's eigenvalue interlacing theorem, we have that for all $k\in[N]$, $$\lambda_k(G) \leq \lambda_k(T)\leq \lambda_{k+N}(G).$$ Let $f:\R_+\to \R_+$ be a bounded non-decreasing function and $f_\star ( \cdot): = f(|\cdot|)$. Then \begin{align}\label{thm 1.3 eqn 1}
\int f_\star d\mu_T \leq 2 \int f_\star d\mu_G. 
\end{align}
Indeed,
\begin{align*}
\int f_\star(t) \mu_T(dt) & = \frac{1}{N}\sum_{k=0}^{N-1}f(|\lambda_k(T)|) \leq \frac{1}{N}\sum_{k=0}^{N-1}f(\max(|\lambda_k(G)|, |\lambda_{k+N}(G)|)) \\
& \leq 2\cdot \frac{1}{2N}\sum_{k=0}^{2N-1}f(|\lambda_k(G)|) = 2\int f_\star(t) \mu_G(dt).
\end{align*}

From Theorems \ref{main_thm} and \ref{circular_lim_dist} respectively, we obtain that $\mu_T \Rrightarrow \nu_\ct$  and $\mu_G \Rrightarrow \nu_\cc$. Thus, by the Skorokhod representation theorem, there exist (separate) probability spaces on which $\mu_T \ed \mu_T'$, $ \nu_\ct \ed \nu_\ct'$, $\mu_T' \Rightarrow \nu_\ct'$ almost surely and  $\mu_G \ed \mu_G'$, $ \nu_\cc \ed \nu_\cc'$, $\mu_G' \Rightarrow \nu_\cc'$ almost surely. In particular, almost surely we have 
\begin{align} \label{thm 1.3 eqn 2}
\int f_\star d\mu_T' \to \int f_\star d\nu_\ct' \qquad \text{and}\qquad\int f_\star d\mu_G' \to \int f_\star d\nu_\cc'.
\end{align} 
Thus, from \eqref{thm 1.3 eqn 1} and \eqref{thm 1.3 eqn 2} we have that \begin{align}\label{thm 1.3 eqn 3}
\int f_\star d\nu_\ct \preccurlyeq 2\int f_\star d\nu_\cc,
\end{align} where $\preccurlyeq$ denotes (first-order) stochastic domination.

Fix $\beta>0$, let $f^R(t): = \max(e^{\beta t}, R)$ and let $f^\infty(t): = e^{\beta t}$ for $t>0$. Then, by \eqref{thm 1.3 eqn 3}, we have \begin{align}\label{thm 1.3 eqn 4}
\int f^R_\star d\nu_\ct \preccurlyeq 2\int f^R_\star d\nu_\cc, \quad \text{ for each } R>0.
\end{align}
By the monotone convergence theorem, we have that almost surely \begin{align}\label{thm 1.3 eqn 5}
\int f_\star^R d\nu_\ct \uparrow \int f_\star^\infty d\nu_\ct\qquad \text{and}\qquad\int f_\star^R d\nu_\cc \uparrow \int f_\star^\infty d\nu_\cc,
\end{align}
as $R \to \infty$. Thus, from \eqref{thm 1.3 eqn 4} and \eqref{thm 1.3 eqn 5} we obtain that $$\int f^\infty_\star d\nu_\ct \preccurlyeq 2\int f^\infty_\star d\nu_\cc.$$
Therefore, it suffices to show that $$\int f_\star^\infty d\nu_\cc \leq 2\exp\Bigl(2\beta^2\sum_{j=0}^\infty \Gamma_j^{-2/\alpha}\Bigr).$$ Note that \begin{align*}
\int f_\star^\infty d\nu_\cc  = \int e^{\beta |t|} \nu_\cc(dt)\leq \int e^{\beta t} \nu_\cc(dt) + \int e^{-\beta t} \nu_\cc(dt),
\end{align*}
so it suffices to upper bound each of the terms on the right of the above display. We do this for the first term; the proof for the second term just follows by replacing $\beta$ by $-\beta$.

For $j \geq 0$, let $\theta_j:=|\cos(2\pi U_j)|$ and $\eta_j: = \sgn( \cos(2\pi U_j))$. We note that the sequences $\eta = (\eta_j)_j$ and $\theta = (\theta_j)_j$ are independent of each other, with $(\eta_j)_j$ being i.i.d.\ Rademacher random variables. Denoting by $\E_\eta$ and $\E_\theta$ the expectations in those variables only, we have \begin{align*}
	\int e^{\beta t} \nu_\cc^\omega(dt) & = \E^\omega \exp\Bigl(2\beta\sum_{j=0}^\infty \Gamma_j^{-1/\alpha} \cos(2\pi U_j)\Bigr)  = \E_\theta\E_\eta \exp\Bigl(2\beta\sum_{j=0}^\infty \Gamma_j^{-1/\alpha} \eta_j\theta_j \Bigr)\\ 
	& \leq \E_\theta\liminf_{n\to\infty} \E_\eta \exp\Bigl(2\beta\sum_{j=0}^n \Gamma_j^{-1/\alpha} \eta_j\theta_j \Bigr) = \E_\theta \prod_{j=1}^\infty \cosh(2\beta\Gamma_j^{-1/\alpha}\theta_j),
\end{align*}
where the inequality follows from Fatou's lemma. By the inequality $\cosh(x) \le e^{x^2/2}, x \in \R$ and the fact that $\theta_j \leq 1$, the above is bounded by \begin{align*}
\E_\theta \prod_{j=0}^\infty \cosh(2\beta\Gamma_j^{-1/\alpha} \theta_j) \leq \exp\Bigl(2\beta^2\sum_{j=0}^\infty \Gamma_j^{-2/\alpha}\Bigr).
\end{align*}

\subsection{Proof of part (d)}
The proof involves two steps. First, we show that $\nu_{\ct}$ has unbounded support with positive probability. Then we boost that probability to one by appealing to a zero-one law that makes use of the ergodic nature of the operator $\Delta$.

\subsubsection{Unboundedness of the support of \texorpdfstring{$\nu_\ct$}{nu T} with positive probability}

To prove the first step, we begin by showing that the limiting spectral measure of the circulant matrix $G$ has unbounded support with probability one.
\begin{lemma}\label{lem circ_supp_unbdd}
For all $\omega \in \Omega_0$, the support of $\nu^\omega_\cc$ is unbounded.
\end{lemma}

\begin{proof}
Since $\omega \in \Omega_0$ and $1 \leq \alpha < 2$, the sequence $(\Gamma_j)_{j \ge 0}$ satisfies $\sum_{j=0}^\infty \Gamma_j^{-1/\alpha} = \infty$. Therefore, given any $R>0$, we can choose $K = K(\omega)$ such that $\sum_{j=0}^{K-1} \Gamma_j^{-1/\alpha} \geq 2R$. This implies that
\begin{align}\label{eq:first_K-1}
\P^\omega\Bigl(\sum_{j=0}^{K - 1} \Gamma_j^{-1/\alpha}\cos(2\pi U_j) \geq R \Bigr) > 0.
\end{align}
By the symmetry of the distribution of $(\cos(2\pi U_j))_{j\geq K}$ about zero, we have
\begin{align}\label{eq:from_K}
\P^\omega\Bigl(\sum_{j=K}^\infty \Gamma_j^{-1/\alpha}\cos(2\pi U_j) \ge 0 \Bigr) = \frac{1}{2}.
\end{align}
Since $(U_j)_{j\geq 0}$ are independent, the events in \eqref{eq:first_K-1} and \eqref{eq:from_K} are independent. This yields 
\begin{align*}
\nu^\omega_\cc ([R, \infty)) &= \P^\omega\Bigl(\sum_{j=0}^\infty \Gamma_j^{-1/\alpha}\cos(2\pi U_j) \ge R \Bigr)\\
&\ge  \P^\omega \Bigl(\sum_{j=0}^{K-1} \Gamma_j^{-1/\alpha} \cos(2\pi U_j)  \geq R \Bigr) \P^\omega \Bigl(\sum_{j=K}^\infty \Gamma_j^{-1/\alpha} \cos(2\pi U_j) \ge 0 \Bigr) > 0,
\end{align*}
which implies that $\sup(\supp(\nu^\omega_\cc)) \geq R$ for each $R>0$, proving the lemma.
\end{proof}

On the probability space $(\Upsilon, \mathcal{G},\mathbf{P})$, define the event 
\begin{equation}\label{eq:E_tail}
\cE = \{\omega \in \Omega_0: \inf\{R>0:\nu_\ct^\omega([-R,R]) = 1\} = \infty \}.
\end{equation}
\begin{lemma} \label{lem positive_prob}
 $\mathbb{P}(\cE) >0$ when $1 \le \alpha < 2$.
\end{lemma}
\begin{proof}
We write down the circulant matrix $G = G_{2N}$ as \begin{align*}
G_{2N} = \begin{pmatrix}
T_N & S_N\\
S_N & T_N
\end{pmatrix}
\end{align*}
where $T_N$ and $S_N$ are symmetric Toeplitz matrices whose first rows are given by $(b_0, \ldots, b_{N-1})$ and $(b_N, \ldots, b_{2N-1}) = (b_N, b_{N-1}, \ldots, b_1)$ respectively. As remarked in Section \ref{sec 2.1}, we have the freedom to choose $b_N$, and for this lemma, we shall assume it to be an independent copy of $b_0$ so that $T_N \ed S_N$.

Following \cite{Fer}, we define $H_{2N}$ to be the unitary matrix given by \begin{align*}
H_{2N} = \frac{1}{\sqrt{2}}\begin{pmatrix}
I_N & I_N\\
I_N & -I_N
\end{pmatrix}
\end{align*} and note that \begin{align} \label{eq:G_N_rep}
H_{2N}^*G_{2N}H_{2N} = \begin{pmatrix}
T_N + S_N & 0_N\\
0_N & T_N - S_N
\end{pmatrix}. 
\end{align}
By Theorem \ref{main_thm}, both $\mu_{T_N} \Rrightarrow \nu_\ct$ and $\mu_{S_N} \Rrightarrow \nu_\ct$. So, using the Skorokhod representation theorem, we can construct two probability spaces on which $\mu_{1,N} \Rightarrow \nu_1$  and $\mu_{2,N} \Rightarrow \nu_2$ almost surely where $\mu_{1,N} \ed \mu_{T_N},$  $ \mu_{2,N}  \ed \mu_{S_N}$ and $\nu_1 \ed \nu_2 \ed \nu_\ct $. With a slight abuse of notation, we will denote a generic element in both probability spaces by $\varpi$.

 Let $R_j(\varpi) = \inf\{R>0:\nu_j^\varpi([-R,R]) = 1\}$ for $j = 1,2$. Let us assume, if possible, that $\mathbb{P}(\cE)=0$. Then we can find $R>0$ such that 
\begin{align*}
    \P (R_1(\varpi) < R) = \P(R_2(\varpi) < R) \geq 3/4.
\end{align*}
If for a fixed $\varpi$,  $\mu^{\varpi}_{1,N} \Rightarrow \nu^{\varpi} $, then  we have $\liminf_N \mu^{\varpi}_{1,N} ((-R, R)) \ge \nu^{\varpi}_1((-R,R))$. Further, if $R_1(\varpi) < R$, then $\nu^{\varpi}_1((-R,R)) =1$. 
Therefore, 
\begin{align*}
\mathbb{P}(\liminf_N \mu^{\varpi}_{1, N} ((-R, R)) = 1) \ge \P (R_1(\varpi) < R) \ge  3/4.
\end{align*}
Since $\mu_{1, N}$ and $\mu_{T_N}$ share the same law, we also have
\begin{align*}
\mathbb{P}(\liminf_N \mu_{T_N} ((-R, R)) = 1) \ge 3/4,
\end{align*}
and the same holds true for $\mu_{S_N}$. So, if we define the event
\[\cE_0 = \big \{\liminf_N \mu_{T_N} ([-R, R]) = 1 \big\} \cap \big\{  \liminf_N \mu_{S_N} ([-R, R]) = 1 \big \},\] then by a union bound,
\begin{equation}\label{eq:TandS_bound}
\mathbb{P}(\cE_0) \ge 1/2.
\end{equation}
If $\cE_0$ holds, it then follows from the Weyl's inequalities on the eigenvalues of the matrices $T_N + S_N$ and $T_N - S_N$ that 
\[ \liminf_N \mu_{T_N+S_N} ([-2R, 2R]) = 1 \qquad \text{and} \qquad \liminf_N \mu_{T_N - S_N} ([-2R, 2R]) = 1.\]
Consequently, it follows from \eqref{eq:G_N_rep} and \eqref{eq:TandS_bound} that 
\[ \mathbb{P} ( \liminf_N \mu_{G_{2N}} ([-2R, 2R]) = 1 ) \ge \mathbb{P}(\cE_0) \ge 1/2.\]
However, since $\mu_{G_{2N}} \Rrightarrow \nu_\cc$ (Theorem \ref{circular_lim_dist}), we obtain a contradiction to the fact that the support of $\nu_\cc$ is unbounded with probability one (Lemma \ref{lem circ_supp_unbdd}). This implies that $\mathbb{P}(\cE) >0,$ as claimed.
\end{proof}

\subsubsection{Almost sure unboundedness of the support of \texorpdfstring{$\nu_\ct$}{nu T}}
We  extend the definition of $\nu_\ct^\omega$ to the entire sample space by setting $\nu^\omega_\ct = \delta_0$ for $\omega \not \in \Omega_0$. Recall that for each $\omega \in \Omega_0$, the operator $\Delta^\omega$ is ergodic (Lemma \ref{lem:ergodic}). It follows (see \cite[Proposition 5.12]{Kir}) that for each $\omega \in \Omega_0$, 
\begin{equation} \label{eq:ergodic_support}
\mathrm{supp}(\nu_\ct^\omega) = \sigma(\Delta^\omega).
\end{equation}

\begin{lemma}\label{lem 6.2}
Let $(E_j)_{j\geq 0}$ be i.i.d.\ $\mathrm{Exp}(1)$ random variables  on $(\Upsilon, \mathcal{G},\mathbf{P})$ such that  $\Gamma_j  =  E_0 + \ldots + E_j$ for each $j$. Then, the event $\cE$, defined in \eqref{eq:E_tail}, is measurable with respect to the exchangeable  $\sigma$-algebra generated by the i.i.d.\ random vectors $(E_j, \zeta_j)_{j\geq 0}$ on $(\Upsilon, \mathcal{G},\mathbf{P})$.
\end{lemma}

\begin{proof}
Let $\tau: \Z_+ \to \Z_+$ be a bijection that keeps all but finitely many indices fixed. Let $n = \max\{ j \ge 0: \tau(j) \ne j\}$, which is finite by our assumption. Denote the sum $ E_{\tau(0)} + E_{\tau(1)}+ \ldots + E_{\tau(j)} $ by $\Gamma^\tau_j$. Obviously, $\Gamma^\tau_j  = \Gamma_j$ and $\zeta_{\tau(j)} = \zeta_j$ for all $j > n$.  Let  $\varrho_k^{\omega, \tau}$ be the obtained from $\varrho_k^\omega$ by replacing $\omega = (\Gamma_j, \zeta_j)_{j\geq 0}$ with $\omega^\tau: = (\Gamma_j^\tau, \zeta_{\tau(j)})_{j\geq 0}$ while keeping the $U = (U_j)_{j\geq 0}$ unchanged. Similarly, we define $\Delta^{\omega, \tau}$ and $\nu_\ct^{\omega, \tau}$. Now by \eqref{eq:ergodic_support}, for all $\omega \in \Omega_0$, we have 
\begin{align}
	\mathrm{dist}( \mathrm{supp}(\nu_\ct^\omega), \mathrm{supp}(\nu_\ct^{\omega, \tau}) )  &=  \E^\omega	\mathrm{dist}(\sigma(\Delta^\omega)) , \sigma(\Delta^{\omega, \tau})) \nonumber \le  \E^\omega \|\Delta^\omega - \Delta^{\omega, \tau}\|_{\op}
\end{align}
where the inequality above can be found in \cite[(A.14)]{Aiz}. Since $\|\Pi\|_\op = 1$, we can further bound the RHS from above as 
\[ \E^\omega \|\Delta^\omega - \Delta^{\omega, \tau}\|_{\op} \leq \sup_{ U } \sup_{ k \in \Z} | \varrho^\omega_k  -   \varrho^{\omega, \tau}_k | \le 2  \sum_{j=0}^n \big (  \Gamma_j^{-1/\alpha} + (\Gamma_j^\tau)^{-1/\alpha} \big ) < \infty.\]
Note that if $\omega \not \in \Omega_0$ then $\omega^\tau \not \in \Omega_0$. In this case, $\mathrm{supp}(\nu_\ct^\omega) =  \mathrm{supp}(\nu_\ct^{\omega, \tau}) = \{0\}$ trivially. Therefore, we have shown that $\mathrm{dist}( \mathrm{supp}(\nu_\ct^\omega), \mathrm{supp}(\nu_\ct^{\omega,\tau}) )  < \infty$ for all realizations of $\omega$. The lemma then follows immediately. 
\end{proof}

From the above lemma and the Hewitt-Savage zero-one law we have that $\P(\cE) \in \{0,1\}$. It follows from Lemma \ref{lem positive_prob} that $\P(\cE) = 1.$


\bibliographystyle{siam}
\bibliography{htt_ref1}
\end{document}